\documentclass[11pt]{amsart}
\usepackage{amsmath,amsthm,amsfonts,amssymb,bbm, setspace,graphicx,float,subfigure,xcolor,mathrsfs,mathtools, Style}
\usepackage{todonotes}
\usepackage[T1]{fontenc}
\usepackage[utf8]{inputenc}
\usepackage{enumitem}
\setlist[enumerate,1]{label={(\roman*)}}

\begin{document}

\author{Lasse L. Wolf}
\address{Institut des Hautes Études Scientifiques, 35 Route de Chartres,
91440 Bures-sur-Yvette} \email{\href{mailto:wolf@ihes.fr}{wolf@ihes.fr}}

\newcommand{\revision}[1]{{\color{blue} #1}}
\newcommand{\Lasse}[1]{{\color{red} \sf $\clubsuit\clubsuit\clubsuit$ Lasse: [#1]}}

\title[Spectra of elliptic invariant DO]{Spectra of elliptic invariant differential operators
on locally symmetric spaces}

\begin{abstract}
	For a locally symmetric space $\Gamma\bk G /K$ we study the spectrum of the invariant differential operators of $G/K$ on $L^2(\Gamma\bk G/K)$.
	This gives a purely analytical proof of the recent theorem
	\cite[Thm.~1.4]{LWW}
	that the quasiregular representation $L^2(\Gamma \bk G)$
	is tempered
	if the limit cone $\mathcal{L}_\Gamma$ of $\Gamma$ is contained in
	the interior of the Weyl chamber
	except if $G$ has real rank one
	or is locally isomorphic to $\mathfrak{sl}_3(\mathbb K)$, $\K=\R,\C,\mathbb H$, or $\mathfrak{e}_6^{-26}$.
\end{abstract}

\subjclass[2020]{22E40, 22E46, 58C40}

\setcounter{tocdepth}{1} 

\maketitle

\section{Introduction}
Let $G$ be semisimple Lie group with finite center,
$K$ a maximal compact subgroup of $G$,
and $G=K\exp(\mathfrak{a}_+)K$ a Cartan decomposition,
where $\mathfrak{a}_+ \subseteq \mathfrak{a}\simeq \R^r$
is a (closed) positive Weyl chamber.
Let $\mu_+ \colon G \to \mathfrak{a}_+$ be the Cartan projection,
$W$ the Weyl group,
and $\mathfrak{a}_+^\ast$ the Weyl chamber in $\mathfrak{a}^\ast$
corresponding to $\mathfrak{a}_+$.
For $\Gamma$ a discrete subgroup of $G$
we denote by $\wt \sigma$ the joint spectrum of the algebra $\mathbb{D}(G/K)$
acting on $L^2(\Gamma \backslash G/K)$.
Here, $\wt \sigma$ is contained in $\mathfrak{a}^\ast_\C$,
the complexified dual of $\mathfrak{a}$.
Furthermore, let $\rho\in \mathfrak{a}^\ast_+$ be the half sum of positive restricted roots
counted with multiplicities.
For $\Gamma$ Zariski dense we consider the modified critical exponents%
\footnote{
	In \cite{LWW} the definition of $\delta'_\lambda$ is given in terms
	of the growth indicator function.
	The two definitions agree if $\Gamma$ is Zariski dense or if $\lambda$ is positive on the limit cone $\mathcal{L}_\Gamma$.}
\[
	\delta'_\lambda \coloneqq
	\inf\left\{s\in \R\colon
	\sum_{\gamma\in \Gamma} e^{-(s\lambda+\rho)\mu_+(\gamma)}<\infty\right\},
	\quad \lambda\in \mathfrak{a}^\ast_+.
\]

Recently, C. Lutsko and T. Weich together with the author
proved the following relation between the joint spectrum $\wt \sigma$
and the critical exponents of $\Gamma$.
We will here present a purely analytic proof of their result.
{\renewcommand{\thetheorem}{A}
	\begin{theorem}	[{\cite[Thm.~1.1]{LWW}}]
		\label{thmA}
		Let
		\[
			\theta_\lambda \coloneqq
			\inf \{\theta \geq 0\colon \Re \wt \sigma \subseteq \theta \operatorname{conv} (W\lambda)\},
			\quad \lambda\in \mathfrak{a}^\ast.
		\]
		Then,
		if $\lambda\in \mathfrak{a}_+^\ast$ with $-\lambda \in W\lambda$,
		it holds
		\[
			\theta_\lambda = \max(0,\delta_\lambda').
		\]
	\end{theorem}
}

This theorem is then used to (almost completely) prove a conjecture of Kim, Minsky, and Oh \cite{KMO24}.
{\renewcommand{\thetheorem}{B}
	\begin{theorem}
		[{\cite[Thm.~1.4]{LWW}}]
		\label{thmB}
		Assume that $G$ has real rank at least $2$
		and is not locally isomorphic
		to $\mathfrak{sl}_3(\mathbb{K})$, $\mathbb{K} = \R,\C, \mathbb{H}$,
		or $\mathfrak{e}_6^{-26}$.
		If the limit cone
		\begin{align*}
			\mathcal{L}_\Gamma = 
			\{v\in \mathfrak{a} \colon
				|\mu_+(\Gamma) \cap \mathcal{C}| = \infty
				\text{ for all open cones } \mathcal{C} \subseteq \mathfrak{a}_+
			\text{ containing } v\}
		\end{align*}
		is contained in the interior $\mathfrak{a}_{++}$ of $\mathfrak{a}_+$,
		then
		\[
			\wt \sigma \subseteq i \mathfrak{a}^\ast.
		\]
	\end{theorem}
}

We refer to the introduction and references in \cite{LWW} for an in-depth discussion of 
their results.

Let us shortly outline the proof of Theorem~\ref{thmA}
given in \cite{LWW}
which relies heavily on representation theoretic techniques for the unitary quasi-regular representation $R$ of $G$ on $L^2(\Gamma \backslash G)$.

Firstly, $\theta_\lambda$ is connected to the decay of matrix coefficients in $L^2(\Gamma \backslash G)$.
More precisely,
for $\lambda\in \mathfrak{a}^\ast_+$,
it is shown that
\[
	\Re \wt \sigma \subseteq  \operatorname{conv} (W\lambda)
\]
if and only if
for all $\varepsilon >0$ there is a constant $C$ such that,
for all $f_1,f_2\in L^2(\Gamma \backslash G/K)$,
\begin{equation}
	\label{eq:estL2}
	|\langle R(g) f_1, f_2\rangle_{L^2(\Gamma \backslash G)}| \leq 
	C e^{\varepsilon \|\mu_+(g)\|} e^{(\lambda - \rho)\mu_+(g)}
	\|f_1\| \| f_2\|
	\quad \forall g\in G.
\end{equation}
This is done by abstract direct integral decomposition of $L^2(\Gamma \backslash G)$
into irreducible unitary representations of $G$.
The irreducible representations needed to analyze the matrix coefficients for $K$-invariant functions $f_1,f_2$
are $K$-spherical
and
have precisely controlled decay
which leads to the above equivalence.

Secondly,
for compactly supported continuous functions $f_1,f_2\in C_c(\Gamma \backslash G)$
they obtain by elementary methods
the existence of $C=C(f_1,f_2,\varepsilon)$ such that
\begin{equation}
	\label{eq:estCcfunctions}
	|\langle R(g) f_1, f_2\rangle_{L^2(\Gamma \backslash G)}| \leq 
	C e^{\varepsilon \|\mu_+(g)\|} e^{(\delta'_\lambda \lambda - \rho)\mu_+(g)}
	\quad \forall g\in G.
\end{equation}
The constant $C$ in the estimate depends on the functions $f_1,f_2$
and it is not at all clear that the dependence is just by their $L^2$-norm
so that the first part of the proof is not yet ready to be applied.
However, \eqref{eq:estCcfunctions}
can be upgraded to \eqref{eq:estL2} with $\lambda$ replaced by%
\footnote{We note that there might be some slight ambiguity in the definition of
	$\delta'_\lambda$ if $\ker \lambda$ intersects $\mathcal{L}_\Gamma$.
	For $\lambda\in \mathfrak{a}_+^\ast$ this can only happen if $\mathfrak{g}$
is not simple.}
 $\delta'_\lambda \lambda$ using 
 the results of \cite{Cow23}
 that were used as a black box in \cite{LWW}
 that involve heavy arguments from the theory of $C^\ast$-algebras.
The main motivation of starting the project leading to the present article
was to find a self-contained proof relying on the analysis of the invariant differential operators $\mathbb{D}(G/K)$.

The previously known proof of Theorem~\ref{thmA} in the case
where $G$ has real rank one \cite{pattersonlimitset,MR360472,MR360473,MR360474,Cor90}
followed a completely different approach.
There, $\mathbb{D}(G/K)$ is generated by the Laplace-Beltrami operator $\Delta_{G/K}$
and
\[
	\wt \sigma = \{\lambda\in \C \colon -\lambda^2 + \rho^2 \in \sigma(\Delta_{\Gamma \backslash G/K})\}
	.
\]
In order to control $\wt \sigma$ one considers the resolvent kernel $k(z,gK,hK)$
of $(\Delta_{G/K} -z) ^{-1}$ for $\Re z<\rho^2$.
For every small $\varepsilon>0$ one has the estimate
\[
	|k(z,gK,hK)|\leq C_\varepsilon e^{\left(-\sqrt{\rho^2 - \Re z} -\rho +\varepsilon\right)
	d(gK,hK)}.
\]
The Schwartz kernel of resolvent of $\Delta_{\Gamma \backslash G/K}$ is given
by the $\Gamma$-average
\[
	\sum_{\gamma\in \Gamma} k(z,\gamma gK,hK).
\]
Thus, as soon as
\[
	\sqrt{\rho^2-z} +\rho>  \delta = 
	\inf\left\{s\in \R\colon
	\sum_{\gamma\in \Gamma} e^{-s d(\gamma g,h)}<\infty\right\}
\]
one infers $z\notin \sigma(\Delta_{\Gamma \backslash G/K})$
and $\max \Re \wt \sigma = \max(\delta-\rho,0)$ follows.
This is Theorem~\ref{thmA} for $\lambda=\rho$ in the rank one case.

The same arguments can be applied in higher rank 
to determine the spectrum of the Laplace-Beltrami operator \cite{AZ22,Web08,Leu04}.
Here one uses the celebrated heat kernel estimates by Anker and Ji \cite{AJ99}. 
However, the Laplace-Beltrami operator does not determine the joint spectrum $\wt \sigma$
and thus their results do not give Theorem~\ref{thmA}.

In this article we give a proof of the inequality $\theta_\lambda \leq \delta'_\lambda$
of Theorem~\ref{thmA},
i.e.~we prove
\begin{equation}
	\label{eq:resigmainconv}
	\Re \wt	\sigma \subseteq \delta'_\lambda \operatorname{conv} (W\lambda)
\end{equation}
for $\lambda\in \mathfrak{a}_+^\ast$ with $-\lambda \in W\lambda$,
using the same line of arguments as in the rank one case.
To deal with the whole algebra $\mathbb{D}(G/K)$ we first observe that
(Lemma~\ref{la:defjointspectrumelliptic})
\[
	\wt \sigma = \{\lambda\in \mathfrak{a}_\C^\ast 
		\colon \chi_\lambda(D) \in \sigma(_\Gamma D)
	\text{ for all elliptic and self-adjoint } D\in \mathbb{D}(G/K)\}
\]
where ${}_\Gamma D$ is the corresponding operator on $\Gamma \backslash G/K$
and $\chi_\lambda(D)$ is the symbol associated with $D$ (see Section~\ref{sub:Invariant differential operators and the Harish-Chandra isomorphism}).
Thus, the joint spectrum $\wt \sigma$ is determined by the individual spectra of elliptic
self-adjoint invariant differential operators on $\Gamma \backslash G/K$.
We then analyze these spectra independently as above:
We first determine the resolvent kernel (Proposition~\ref{la:kernelformulaGK})
and prove resolvent estimates (Proposition~\ref{prop:decaykernel}).
It follows that (Theorem~\ref{thm:spectrumsinglesa})
\begin{equation}
	\label{eq:thmspeccontained}
	\sigma(_\Gamma D) \subseteq \chi_{[0,\max(0,\delta'_\lambda)] \lambda +i \mathfrak{a}^\ast}(D)
\end{equation}
for every elliptic self-adjoint $D\in \mathbb{D}(G/K)$ with high enough order.
Thus,
\[
	\wt \sigma \subseteq\bigcap_{p\in \mathcal{P}_\R} \{ \lambda\in \mathfrak{a}_\C^\ast \colon
	p(i\lambda) \in p(\mathfrak{a}^\ast + i [0,\max(0,\delta'_\lambda)] \lambda)\}
\]
where the intersection is over the set $\mathcal{P}_\R$ of
elliptic $W$-invariant polynomials $p$ on $\mathfrak{a}_\C^\ast$
with real coefficients (Proposition~\ref{prop:jointspecinintersectionofpolys}).
It remains to determine the $\mathcal{P}_\R$-hull
of $\mathfrak{a}^\ast + i [0,1]\lambda$ for $\lambda\in \mathfrak{a}^\ast$.
In Section~\ref{sec:Separating tubes by invariant polynomials}
we prove that (Theorem~\ref{thm:realpolynomialhull})
\[
	\bigcap_{p\in \mathcal P_\R} \{ \lambda\in \mathfrak{a}_\C^\ast \colon
	p(i\lambda) \in p(\mathfrak{a}^\ast + i [0,1] \lambda)\}
	\subseteq  \operatorname{conv}(W\lambda \cup W(-\lambda)) + i\mathfrak{a}^\ast
\]
which yields by Lemma~\ref{la:nonhermitianconvexhull}
\begin{equation}
	\label{eq:bdspecnonHerm}
	\Re \wt \sigma \subseteq \max(0,\delta'_\lambda)\operatorname{conv}\left(W
	\frac{\lambda-w_0 \lambda}{2}\right),
\end{equation}
where $w_0$ is the unique element in $W$
mapping $\mathfrak{a}_+$ to $-\mathfrak{a}_+$.
In particular, we get \eqref{eq:resigmainconv}
if $-\lambda\in W\lambda$.
\begin{remark}
	\begin{enumerate}
		\item In our proof we have to assume that $\mathcal{L}_\Gamma$ is contained in $\mathfrak{a}_{++}$.
			This is due to the reason that we prove the resolvent estimates only in $\mathfrak{a}_{++}$.
			However,
			it is likely that they hold in a similar fashion near the boundary.
			The main problem is that the function $\Phi_\lambda$ determining the resolvent kernel
			is only defined in $\mathfrak{a}_{++}$.
			Grouping different $\Phi_{w\lambda}$, $w\in W$,
			corresponding to different walls of the Weyl chamber
			should resolve this problem as in \cite{CM82}.
			As one assumes $\mathcal{L}_\Gamma \subseteq \mathfrak{a}_{++}$ in
			Theorem~\ref{thmB},
			we get a new proof of it and
			we decided not to push the method of grouping to the limit.
		\item For non-Hermitian $\lambda$ (i.e.~$-\lambda\notin W\lambda$)
			there is no upper bound on $\theta_\lambda$
			written in \cite{LWW}.
			However, analyzing carefully the methods of \cite{Cow23}
			(which need the assumption $-\lambda \in W\lambda$)
			one also finds \eqref{eq:bdspecnonHerm}.
		\item Since \eqref{eq:thmspeccontained}
			is not sharp as remarked in Section~\ref{sub:specindividual}
			we do not get directly that $\max(0,\delta'_\lambda)$
			is the optimal constant
			bounding $\Re \wt \sigma$.
			The optimality (i.e.~$\theta_\lambda=\max(0,\delta_\lambda')$)
			will be proved in Section~\ref{sub:Completion of the proof for thmA}
			using the exact value of the bottom of the Laplace spectrum obtained in \cite{AZ22}
			under the additional assumption that $\Gamma$
			is Zariski dense.
			To summarize, this paper gives a new proof
			of Theorem~\ref{thmA}
			for Zariski dense $\Gamma$ with $\mathcal{L}_\Gamma\subseteq \mathfrak{a}_{++}$.
	\end{enumerate}
\end{remark}
In the last section (Section~\ref{sec:prooftempered})
we consider the domain of convergence
\[
	R=\operatorname{closure} (\{\lambda \in\mathfrak{a}^\ast \colon P_\lambda <\infty\})
\]
for the Poincar\'e series
\[
	P_\lambda = \sum_{\gamma\in \Gamma} e^{-\lambda\mu_+(\gamma)}
	\qquad \text{for }\lambda\in \mathfrak{a}^\ast
	.
\]
Recall that there is an
optimal element $\mu_\Gamma$ in $\mathfrak{a}^\ast$
bounding $\Re \wt \sigma$ (see \cite[Prop.~5.1]{LWW}).
This element is used to get a precise description of $R$.
By definition, $\mu_\Gamma + \rho \in R$
and thus $\mu_\Gamma + \rho + \mathcal{L}_\Gamma^\star \subseteq R$
where $\mathcal{L}_\Gamma^\star$ is the dual cone of $\mathcal{L}_\Gamma$.
Using Theorem~\ref{thmA} we also get an upper bound on $R$, see Proposition~\ref{prop:upperboundR'}.
This leads to a new proof of Theorem~\ref{thmB}, see \eqref{eq:onewallavoided}.
Let us here give the description for $R$ in the case where $\mu_\Gamma$ is regular
and refer to Proposition~\ref{prop:upperboundR'} for the general statement.
\begin{theorem}
	Assume that $\mu_\Gamma \in \mathfrak{a}_{++}^\ast$ (the set corresponding to $\mathfrak{a}_{++}$ under the identification of $\mathfrak{a}$ and $\mathfrak{a}^\ast$ by the Killing form),
	that $w_0 =-1$,
	and that $\Gamma$ is Zariski dense so that $\mathcal{L}_\Gamma$ is convex.
	Then,
	\[
		\mathcal{L}_\Gamma = \mathfrak{a}_+
		\quad
		\text{and}
		\quad
		R=\mu_\Gamma +\rho + \mathcal{L}_\Gamma^\star.
	\]

\end{theorem}

\subsection*{Acknowledgements}
I thank Benjamin Delarue and Tobias Weich for their help and advice in many aspects of this work.
I also want to thank Peter Smillie for discussions leading to Section~\ref{sec:prooftempered}.

\section{Fourier-Helgason transform and resolvent estimates on \texorpdfstring{$G/K$}{G/K}}
\label{sec:analysisonGK}
\subsection{Notation for semisimple Lie groups}%
\label{sub:Notation for semisimple Lie groups}
In this article,
$G$ is a real semisimple connected non-compact Lie group with finite center
and $K$ is a maximal compact subgroup of $G$.
The corresponding Riemannian symmetric space of noncompact type is then $G/K$.
We fix an Iwasawa decomposition $G=KAN$, and have $A\cong \R^r$ where $r$ is the real rank of $G$ or the rank of the symmetric space $G/K$.
Furthermore,  we define $M$ as the centralizer of $A$ in $K$ and $\overline{N}$ to be the nilpotent subgroup such that $KA\overline{N}$ is the opposite Iwasawa decomposition.
We denote by fractal letters
the corresponding Lie algebras.
For $g\in G$ let $H(g)\in \mathfrak{a}$ be the logarithm of the $A$-component in the Iwasawa decomposition.
Let $\Sigma \subseteq \mathfrak{a}^\ast$ be the root system of restricted roots,
$\Sigma^+$ the positive system corresponding to the Iwasawa decomposition,
$\Pi$ the subset of simple roots,
and $W$ the corresponding Weyl group acting on $\mathfrak{a}^\ast$.
As usual, for $\alpha\in \Sigma$, we denote
by $m_\alpha$ the dimension of the root space,
and by $\rho$ the half sum of positive restricted roots counted with multiplicity.
Let $\mathfrak{a}_+ =\{H\in \mathfrak{a}\mid \alpha(H)\geq 0 \:\forall \alpha\in \Sigma\}$ the (closed) positive Weyl chamber,
and $\mathfrak{a}_+^\ast$ the corresponding cone in $\mathfrak{a}^\ast$
via the identification $\mathfrak{a} \leftrightarrow \mathfrak{a}^\ast$ through the Killing form $\langle\cdot,\cdot\rangle$.
We have the Cartan decomposition $G=K\exp(\mathfrak a_+)K$
and for $g\in G$ there is a unique $\mu_+(g)\in \mathfrak a_+$ such that
$g\in K \exp(\mu_+(g))K$.

\subsection{Invariant differential operators and the Harish-Chandra isomorphism}%
\label{sub:Invariant differential operators and the Harish-Chandra isomorphism}
In this section we introduce one of the main objects of this article,
namely the algebra $\mathbb D(G/K)$ of \emph{$G$-invariant differential
operators} on $G/K$, i.e. differential operators commuting with the left
regular representation $L_g$ for elements  $g\in G$ where $L_g f(x)\coloneqq f(g ^{-1} x)$.
This algebra always contains the Laplace-Beltrami operator $\Delta$.
In general it is a commutative algebra that can be identified with a set of polynomials by the so-called
Harish-Chandra isomorphism.

\begin{theorem}[see {\cite[Thm.~II.5.17]{gaga}}]\label{thm:HC}
	There is an algebra isomorphism
	\begin{align*}
		\operatorname{HC} \colon \mathbb D(G/K)\to \mathrm{Poly}(\mathfrak a^\ast_\C)^W
	\end{align*}
	from the $G$-invariant differential operators $\mathbb D(G/K)$ to the Weyl group invariant polynomials  on $\mathfrak a_\C^\ast$. We write $\chi_\lambda(D)$ instead of $\operatorname{HC}(D)(\lambda)$
	and $D_p$ for the operator in $\mathbb{D}(G/K)$ 
	with $\chi_\lambda(D_p) = p(i\lambda)$ for
	$p \in \operatorname{Poly} (\mathfrak{a}_\C^\ast)^W$.
\end{theorem}
The convention for $D_p$ is chosen with the multiplication by $i$
in order to have real polynomials correspond to symmetric operators (see \eqref{eq:symmetricDO}).

By Chevalley's Theorem \cite[Thm.~3.5]{Hum90}
$\operatorname{Poly} (\mathfrak{a}_\C^\ast)^W$ and consequently also $\mathbb{D}(G/K)$ 
is a polynomial algebra in $r=\dim \mathfrak{a}$ generators.
The generators of $\operatorname{Poly} (\mathfrak{a}_\C^\ast)^W$
can be written down in terms of the root system.
For example in the case of $\mathfrak{g}=\mathfrak{sl}_n(\R)$
the abelian subalgebra $\mathfrak{a}$ can be chosen as the set of traceless
diagonal matrices.
Then homogeneous generators of $\operatorname{Poly} (\mathfrak{a}_\C^\ast)^W$
are given by the polynomials
\[
	p_i(\operatorname{diag} (\lambda_1,\ldots,\lambda_n))=
	\lambda_1^i + \ldots +\lambda_n^i,
	\quad i=2,\ldots,n.
\]
For the corresponding differential operators $D _{p_i}$ see \cite{Brennecken, phd}.

We close this section with the following lemma.
\begin{lemma}[{\cite[Lemma~5.21, Cor.~5.3]{gaga}}]
	\label{la:HCadjoint}
	Let $D \in \mathbb{D}(G/K)$ and denote by $D^\ast$ its adjoint.
	Then
	\[
		\chi_\lambda(D^\ast)= \overline{\chi _{-\overline{\lambda}}(D)}
		.
	\]
\end{lemma}
\subsection{Spherical functions}%
\label{sub:Spherical functions}
In this section we introduce the spherical functions $\phi_\lambda$
and discuss their Harish-Chandra expansion.
\begin{definition}[]
	The spherical function $\phi_\lambda$ for $\lambda\in \mathfrak{a}_\C^\ast$
	is defined by
	\[
		\phi_\lambda(g)= \int_{K}^{} e ^{-(\lambda+\rho)H(g ^{-1}k)} \: d{k},
		\quad g\in G.
	\]
\end{definition}
We collect some facts about $\phi_\lambda$ in the following proposition.
\begin{proposition}
	[{\cite[Thm.~IV.4.3]{gaga}}]
	\begin{enumerate}
		\item $\phi_\lambda$ is the unique smooth bi-$K$-invariant function $f$
			on $G$
			such that $f(e)=1$ 
			and $D f = \chi_\lambda(D) f$.
		\item $\phi_\lambda=\phi_\mu$ if and only if $\lambda \in W\mu$.
		\item $\phi_\lambda(g ^{-1})= \phi_{-\lambda}(g)$
			and $\overline{\phi_\lambda} = \phi_{\overline{\lambda}}$.
	\end{enumerate}
\end{proposition}

The following expansion of $\phi_\lambda$ will be important to estimate the resolvent kernels.
We first introduce its ingredients.

Let $\Lambda = \N_0 \Pi$ and $\tilde \Lambda = \Z \Pi$.
Define $\Gamma_\mu(\lambda)$ inductively for $\mu\in \Lambda$ by $\Gamma_0(\lambda)=1$
and
\[
	(\langle \mu, \mu \rangle - 2  \langle \mu, \lambda \rangle)\,\Gamma_\mu(\lambda)
	= 2 \sum_{\alpha \in \Sigma^+} m_\alpha \sum_{k \ge 1}
	\Gamma_{\mu - 2k\alpha}(\lambda)
	(\langle \mu + \rho - 2k\alpha, \alpha \rangle -  \langle \alpha, \lambda \rangle)
\]
for $\lambda\notin \sigma_\mu = \{\lambda\in \mathfrak{a}_\C^\ast\colon \langle \mu,\mu\rangle = 2 \langle\lambda,\mu\rangle\}$.
For $\lambda\notin \bigcup _{\mu\in \Lambda \setminus \{0\}} \sigma_\mu$
and in particular if $\Re \lambda \in - \mathfrak{a}_+^\ast$
the function
\[
	\Phi_\lambda (g) = e ^{(\lambda-\rho)(\mu_+(g))} \sum_{\mu\in \Lambda} {\Gamma_\mu(\lambda) e ^{-\mu(\mu_+(g))}}
\]
is well-defined for $\mu_+(g) \in \mathfrak{a} _{++}$ (the interior of $\mathfrak{a}_+$)
by \cite[Lemma~IV.5.3]{gaga}.
The function $\Phi_\lambda$ is built as a perturbation of $e ^{(\lambda-\rho)\mu_+(g)}$
satisfying $D\Phi_\lambda = \chi_\lambda(D) \Phi_\lambda$.
The spherical function $\phi_\lambda$ is a linear combination of $\Phi_{w\lambda}$, $w\in W$.
The coefficients are given by the Harish-Chandra $\mathbf c$-function
which is defined by
\[
	\mathbf c(\lambda)= \int_{\overline N}^{} e ^{-(\lambda+\rho)H(\overline{n})} \: d{\overline{n}} 
	,
	\quad 
	\Re \lambda\in \mathfrak{a}_{++}^\ast.
\]
Here the Haar measure on $\overline{N}$ is normalized by
$\int_{\overline{N}}^{} e ^{-2\rho H(\overline{n})} \: d{\overline{n}} =1$.
The $\mathbf c$-function initially defined on $\Re \lambda\in \mathfrak{a} _{++}^\ast$
has meromorphic continuation to $\mathfrak{a}_\C^\ast$.
In fact, $\mathbf c (\lambda)$ can be written explicitly in terms of $\Gamma$-functions
with arguments depending only on $\langle \lambda,\alpha\rangle$ and $m_\alpha$ for $\alpha\in \Sigma$
\cite[Thm.~IV.6.14]{gaga}.
The expansion of $\phi_\lambda$ now reads as follows.
\begin{theorem}[{\cite[Thm.~IV.5.5]{gaga}}]
	\label{thm:sphericalexpansion}
	If $s\lambda\notin \bigcup_{\mu\in \Lambda\setminus \{0\}} \sigma_\mu$
	and $s\lambda-w \lambda \notin \tilde \Lambda$ for all $s\neq w\in W$,
	then
	\[
		\phi_\lambda(e^H) = \sum_{w\in W} \mathbf c(w\lambda) \Phi_{w \lambda} (e^H)
	\]
	for all $H\in \mathfrak{a} _{++}$.
\end{theorem}

To estimate integrals over $\phi_\lambda$
we will need the following two lemmas.

\begin{lemma}[{\cite[Prop.~IV.7.2]{gaga}}]
	\label{la:estc}
	For suitable constants $C_1$ and $C_2$,
	\[
		|\mathbf c (\lambda)|^{-1} \leq C_1 +C_2 |\lambda|^{\frac 12 \dim N}
	\]
	for $\Re \lambda \in  \mathfrak{a}_+^\ast$.
\end{lemma}

\begin{lemma}
	\label{la:estPhi}
	For every $\varepsilon>0$ there is a constant $C_\varepsilon$
	such that, for every $H$ with $\alpha(H)\geq \varepsilon$ for all $\alpha\in \Pi$,
	and $\Re\lambda \in - \mathfrak{a} _{+}^\ast$,
	\[
		|\Phi_\lambda (e^H) | \leq C_\varepsilon e ^{(\Re \lambda -\rho) (H)}
	\]
\end{lemma}
\begin{proof}
	Let $H_0\in \mathfrak{a}_{++}$ with $\alpha(H_0)=\varepsilon/2$ for all $\alpha\in \Pi$.
	By \cite[Lemma~IV.5.6]{gaga}
	there is a constant $C$ such that
	$|\Gamma_\mu(\lambda)| \leq C e ^{\mu(H_0)}$
	for all $\mu\in \Lambda$ and $\lambda\in - \mathfrak{a}_+^\ast + i \mathfrak{a}^\ast$.
	To prove the lemma we need to see that
	\[
		\sum_{{\mu\in \Lambda}} {\Gamma_\mu(\lambda) e^{-\mu(H)}}
	\]
	is bounded.
	We have $|\Gamma_\mu(\lambda) e ^{-\mu(H)}|\leq C e^{-\mu(H-H_0)}$.
	It follows that
	\begin{align*}
		\left |\sum_{{\mu\in \Lambda}} {\Gamma_\mu(\lambda) e^{-\mu(H)}}\right|
	&	\leq C \sum_{{(n_\alpha)\in \N_0^\Pi}} e^{-\sum_{\alpha\in \Pi} n_\alpha \alpha(H-H_0)}
	\leq C \sum_{{(n_\alpha)\in \N_0^\Pi}} e^{-\sum_{\alpha\in \Pi} n_\alpha \varepsilon/2}
	\\
	&\leq C \sum_{{(n_\alpha)\in \N_0^\Pi}} \prod_{\alpha\in \Pi}e^{- n_\alpha \varepsilon/2}
	\leq C \prod_{\alpha\in \Pi} \sum_{{n_\alpha\in \N_0}i} e^{- n_\alpha \varepsilon/2}
	\\
	&= C \prod_{\alpha\in \Pi} \frac 1 {1-e^{-\varepsilon/2}} \eqqcolon C_\varepsilon
	\end{align*}
	This completes the proof.
\end{proof}

\subsection{Fourier-Helgason transform}%
\label{sub:Fourier-Helgason transform}
The Fourier-Helgason transform gives a spectral resolution of $\mathbb{D}(G/K)$
on $L^2(G/K)$.
Let $e_{\lambda,kM}(gK)=e^{-(\lambda+\rho)H(g^{-1} k)}$
for $\lambda\in \mathfrak{a}_\C^\ast$ and $k\in K$.
Then we have $D e_{\lambda,kM}=\chi_\lambda(D) e_{\lambda,kM}$ by \cite[Lemma~II.5.15]{gaga} for every $D\in\mathbb D(G/K)$.
Thus the functions $e _{\lambda,kM}$ are the replacements on $G/K$
of the exponentials $e ^{\langle\lambda,x\rangle}$ on $\R^n$.
The Helgason-Fourier transform is therefore defined by
\begin{align*}
	\mathcal Ff(\lambda,kM)= \int_{G/K} f(gK)e^{(\lambda-\rho)H(g^{-1} k)}d(gK)
\end{align*}
for a sufficiently nice function $f\colon G/K \to \C$.

\begin{lemma}\label{la:FourierIntertwiner}
	The Fourier-Helgason transform satisfies
	\[
		\mathcal F(Df)(\lambda,kM)=\chi_{\lambda}(D) \mathcal Ff(\lambda,kM)
	\]
	for every $D\in \mathbb D(G/K)$.
\end{lemma}
\begin{proof}
	\begin{align*}\mathcal F(Df)(\lambda,kM)&=\int_{G/K}
		Df(gK)\overline{e_{-\overline\lambda,kM}(gK)}d(gK)=\int_{G/K} f(gK)\overline{D^\ast
		e_{-\overline\lambda,kM}(gK)}d(gK)\\
	&=\int_{G/K}^{} {f(gK) \overline{\chi_{-\overline \lambda}(D^\ast)e_{-\overline\lambda,kM}(gK)}} \: d(gK)=
	\overline{\chi_{-\overline\lambda}(D^\ast)} \mathcal Ff(\lambda,kM).
	\end{align*} 
	By
	Lemma~\ref{la:HCadjoint} $\chi_{\lambda}(D^\ast)=
	\overline{\chi_{-\overline\lambda}(D)}$
	and the lemma follows.
	\end{proof}

	\begin{theorem}[{\cite[Thm.~III.1.5]{HelGeomAna}}]\label{thm:FourierL2}
		The Fourier-Helgason transform is an isometry between $L^2(G/K)$ and $L^2 (i\mathfrak a_+^\ast \times K/M, |\mathbf c(\lambda)|^{-2} d\lambda d(kM))$. Moreover,
		$$\langle f,g\rangle_{L^2(G/K)} = |W|^{-1} \int_{i\mathfrak a^\ast\times K/M} \mathcal Ff(\lambda,kM) \overline {\mathcal Fg(\lambda,kM)} |\mathbf c(\lambda)|^{-2} d\lambda d(kM).
		$$
	\end{theorem}
	\subsection{Resolvent estimates on \texorpdfstring{$G/K$}{G/K}}%
	\label{sub:Resolvent estimates}
	In this section we analyze the resolvent kernels of operators $D\in \mathbb{D}(G/K)$.
	We first determine these kernels by the Fourier-Helgason transform.
	\begin{lemma}
		\label{la:kernelformulaGK}
		Let $D\in \mathbb{D}(G/K)$ and $z\notin \sigma(D) = \{\chi_\lambda(D)\colon \lambda\in i \mathfrak{a}^\ast\}$.
		Then $(D-z) ^{-1}$ is given by the right convolution with
		the $K$-invariant function
		\begin{align*}
			k(D,z,gK) &\coloneqq 
			|W|^{-1} \int_{i \mathfrak{a}^\ast}(\chi_\lambda(D)-z)^{-1}
			\phi_{-\lambda}(g) |\mathbf c(\lambda)|^{-2} d\lambda\\
				  &=
				  \int_{i \mathfrak{a}^\ast}(\chi_{\lambda}(D)-z)^{-1}
				  \Phi_{-\lambda}(h) \mathbf c(\lambda) ^{-1} d\lambda
		\end{align*}
		on $G/K$,
		i.e.~
		\[
			(D-z) ^{-1} f(gK)
			= (f \ast k(D,z))(gK)
			=\int_{G/K}^{} f(hK) k(D,z,h ^{-1}gK) \: d{hK} .
		\]
	\end{lemma}
	\begin{proof}
		The proof is a standard calculation.
		For $f,g \in L^2(G/K)$ we have
		\begin{align*}
			\langle& (D-z)^{-1} f, g\rangle _{L^2(G/K)}=
			|W|^{-1} \int_{i \mathfrak{a}^\ast \times K/M}^{}  \mathcal F (D-z)^{-1}f(\lambda,kM) \overline {\mathcal Fg(\lambda,kM)} |\mathbf c(\lambda)|^{-2} d\lambda d(kM)
			\\
			       &=
			       |W|^{-1} \int_{i \mathfrak{a}^\ast\times K/M}^{} (\chi_\lambda(D)-z)^{-1} \mathcal Ff(\lambda,kM) \overline {\mathcal Fg(\lambda,kM)} |\mathbf c(\lambda)|^{-2} d\lambda d(kM)
			       \\
			       &=
			       |W|^{-1} \int_{i \mathfrak{a}^\ast\times K/M}^{} (\chi_\lambda(D)-z)^{-1}
			       \mathcal Ff(\lambda,kM)
			       \overline {\int_{G/K}^{} {g(hK) e^{(\lambda-\rho) H(h^{-1}k)}} \: d{(hK)} } |\mathbf c(\lambda)|^{-2} d\lambda d(kM)
			       \\
			       &=
			       |W|^{-1} \int_{i \mathfrak{a}^\ast\times K/M}^{} (\chi_\lambda(D)-z)^{-1}
			       \mathcal Ff(\lambda,kM)
			       \int_{G/K}^{} {\overline{g(hK)} e^{(\overline \lambda-\rho) H(h^{-1}k)}} \: d{(hK)}  |\mathbf c(\lambda)|^{-2} d\lambda d(kM)
			       \\
			       &=
			       \int_{G/K}|W|^{-1} \int_{i \mathfrak{a}^\ast\times K/M}^{} (\chi_\lambda(D)-z)^{-1}
			       \mathcal Ff(\lambda,kM)
			       { e^{(- \lambda-\rho) H(h^{-1}k)}}   |\mathbf c(\lambda)|^{-2} d\lambda d(kM)\overline{g(hK)}\: d{(hK)}.
		\end{align*}

		Hence,
		\begin{align*}
			(D&-z)^{-1}f(hK)= 
			|W|^{-1} \int_{i \mathfrak{a}^\ast\times K/M}^{} (\chi_\lambda(D)-z)^{-1}
			\mathcal Ff(\lambda,kM)
			{ e^{(-\lambda-\rho) H(h^{-1}k)}}   |\mathbf c(\lambda)|^{-2} d\lambda d(kM)
			\\
			  &=
			  |W|^{-1} \int_{i \mathfrak{a}^\ast\times K/M}^{} (\chi_\lambda(D)-z)^{-1}
			  \int_{G/K}^{} {f(gK) e ^{(\lambda-\rho)H(g ^{-1}k)}} \: d{(gK)} 
			  { e^{(- \lambda-\rho) H(h^{-1}k)}}   |\mathbf c(\lambda)|^{-2} d\lambda d(kM)
			  \\
			  &=
			  \int_{G/K}^{}f(gK) 
			  |W|^{-1} \int_{i \mathfrak{a}^\ast\times K/M}^{} (\chi_\lambda(D)-z)^{-1}
			  {e ^{(\lambda-\rho)H(g ^{-1}k)}} 
			  { e^{(- \lambda-\rho) H(h^{-1}k)}}   |\mathbf c(\lambda)|^{-2} d\lambda d(kM)\: d{(gK)}.
		\end{align*}
		By \cite[Lemma~IV.4.4]{gaga}
		\begin{align*}
			\int_{K/M}^{}  {e ^{(\lambda-\rho)H(g ^{-1}k)}} 
			{ e^{(- \lambda-\rho) H(h^{-1}k)}} d(kM)
			 &=
			 \phi_\lambda(h ^{-1}g) = \phi _{-\lambda}(g ^{-1}h).
		\end{align*}
		It follows that
		\begin{align*}
			(D-z)^{-1}f(hK)= 
			\int_{G/K}^{}f(gK) 
			|W|^{-1} \int_{i \mathfrak{a}^\ast}^{} (\chi_\lambda(D)-z)^{-1}
			\phi_{-\lambda}(g ^{-1}h)
			|\mathbf c(\lambda)|^{-2} d\lambda d(kM)\: d{(gK)}.
		\end{align*}
		This proves the first equation.

		We now use Theorem~\ref{thm:sphericalexpansion}.
		We note that the assumptions are satisfied for $-\lambda\in i \mathfrak{a}^\ast$
		except if there is $\alpha\in \Pi$ such that $\langle \alpha,\lambda\rangle =0$
		which is a set of measure $0$.
		Hence, we can write
		\begin{align*}
			k(D,z,gK) 
			 &=
			 |W|^{-1} \int_{i \mathfrak{a}^\ast}(\chi_\lambda(D)-z)^{-1}
			 \sum_{{w\in W}} {\mathbf{c}(-w\lambda) }\Phi_{-w\lambda}(g) |\mathbf c(\lambda)|^{-2} d\lambda
			 \\
			 &=
			 |W|^{-1} \sum_{{w\in W}}\int_{i \mathfrak{a}^\ast}(\chi_{w\lambda}(D)-z)^{-1}
			 {\mathbf{c}(-w\lambda) }\Phi_{-w\lambda}(g) |\mathbf c(w\lambda)|^{-2} d\lambda
		\end{align*}
		As $d\lambda$ is $W$-invariant, we infer
		\begin{align*}
			k(D,z,gK) 
			 &=
			 \int_{i \mathfrak{a}^\ast}(\chi_{\lambda}(D)-z)^{-1}
			 {\mathbf{c}(-\lambda) }\Phi_{-\lambda}(g) |\mathbf c(\lambda)|^{-2} d\lambda
			 \\
			 &=
			 \int_{i \mathfrak{a}^\ast}(\chi_{\lambda}(D)-z)^{-1}
			 \Phi_{-\lambda}(g) \mathbf c(\lambda) ^{-1} d\lambda
		\end{align*}
		completing the proof.
	\end{proof}

	We now want to estimate the kernel of Lemma~\ref{la:kernelformulaGK}.
	As usual we will shift the contour of integration to obtain exponential decay.
	The integral is not necessarily absolutely convergent but we will restrict ourselves
	to the case where it is.
	To that end, let us recall the notion of an elliptic polynomial.
	\begin{definition}
		A polynomial $p\in \operatorname{Poly} (\mathfrak{a}_\C^\ast)$
		is called \emph{elliptic}
		if its highest degree homogeneous part $p_m$ satisfies $p_m(\lambda)\neq 0$
		for all $\lambda\in \mathfrak{a}^\ast$.
	\end{definition}
	If now $p$ is elliptic of order $m$,
	then
	\[
		|(p(\lambda)-z) ^{-1} \mathbf{c}(\lambda) ^{-1}| \leq C |\lambda|^{-m + \frac 12 \dim N}
	\]
	by Lemma~\ref{la:estc}
	and thus is integrable if $m>\frac 12 \dim N + \dim A -1$.
	This leads to the following proposition.

	\begin{proposition}
		\label{prop:decaykernel}
		Let $D\in \mathbb{D}(G/K)$
		such that $\chi_\lambda(D)$ is elliptic of order $m$
		with $m > \frac 12 \dim N + \dim A -1$.
		Let $\lambda\in \mathfrak{a}_+^\ast$
		and $z\notin \{\chi_\mu(D) \colon \mu\in [0,1]\lambda + i \mathfrak{a}^\ast\}$.
		Then for any $\varepsilon >0$ there is a constant $C$ such that
		\[
			|k(D,z,e^H)| \leq C e^{(-\lambda-\rho)H}
		\]
		locally uniform in $z$
		for any $H$ with $\alpha(H)\geq \varepsilon$ for all $\alpha\in \Pi$.
	\end{proposition}

	\begin{proof}
		For fixed $\lambda_R\in \mathfrak{a}^\ast$,
		\[
			\chi_{\lambda_R + i \eta}(D) = \chi_{i\eta} + \text{terms of degree $<m$ in $\eta$}
		\]
		is still elliptic in $\eta$.
		Hence, for any compact set $C\subseteq \mathfrak{a}^\ast$,
		\[
			|\chi_\eta (D)| \geq c\|\Im \eta\|^m
		\]
		for $\Re \eta\in C$ and $\|\Im \eta\| \geq R$.

		We complete $\lambda \eqqcolon\lambda_1$ to a basis $\lambda_1,\ldots,\lambda_r$ of $\mathfrak{a}^\ast$.
		Then,
		\begin{align*}
			k(D,z) &= \int_{i \mathfrak{a}^\ast}^{}  (\chi_{\lambda}(D)-z)^{-1}
			\Phi_{-\lambda} \mathbf c(\lambda) ^{-1}\: d{\lambda} 
			\\
			       &=
			       \int_{i \R^n}^{}  (\chi_{\sum x_k \lambda_k}(D)-z)^{-1}
			       \Phi_{-\sum x_k \lambda_k} \mathbf c\left(\sum x_k \lambda_k\right) ^{-1}\: d{x} 
		\end{align*}
		Considering the integrand as a function $f$ of $x_1$,
		we see that $f$ is holomorphic in $(-\epsilon,1+\epsilon) + i\R$
		and $|f(x_1)|\leq C |\Im x_1|^{\frac 12 \dim N -m }$
		for $\Re x_1 \in [0,1]$ and $|\Im x_1| \gg 0$.
		Since $m > \frac 12 \dim N$ we can thus shift the contour of integration to $1 + i\R$
		and get
		\[
			k(D,z) = \int_{\lambda_1 + i \mathfrak{a}^\ast}^{}  (\chi_{\lambda}(D)-z)^{-1}
			\Phi_{-\lambda} \mathbf c(\lambda) ^{-1}\: d{\lambda} .
		\]
		The integrand can now be estimated
		\[
			| (\chi_{\lambda}(D)-z)^{-1}
			\mathbf c(\lambda) ^{-1}|
			\leq 
			C (1 + |\Im \lambda|)^{\frac 12 \dim N -m}
		\]
		for $\lambda \in  \lambda_1 + i \mathfrak{a}^\ast$ as before.
		Thus we have,
		since $\frac 12 \dim N -m + \dim A -1 < 0$,
		\begin{align*}
			|k(D,z,e^H)|\leq C_\varepsilon \int_{\lambda_1+ i \mathfrak{a}\ast}^{} (1+\|\Im \lambda\|)^{\frac 12 \dim N-m}  \: d{\lambda} 
			\cdot
			e^{(-\lambda_1-\rho)H}
			\leq C_\varepsilon 
			e^{(-\lambda_1-\rho)H}
		\end{align*}
		using Lemma~\ref{la:estPhi}.
		Note that the constant is uniform in a small neighborhood of $z$.
		Indeed, if $z\notin \{\chi_\lambda(D) \colon \lambda\in [0,1]\lambda + i \mathfrak{a}^\ast\}$
		then the same holds for a small perturbation of $z$.
		The constant $C$ just comes from the estimate on $|\chi_\lambda(D)-z|$ which is uniform in a small
		neighborhood of $z$.
	\end{proof}

	\section{Separating tubes by invariant polynomials}%
	\label{sec:Separating tubes by invariant polynomials}

	In this section we study how to separate tubes $\mathfrak{a}^\ast + iC$ by $W$-invariant elliptic polynomials.
	We formulate the results in the general language of root systems.
	Let $\Sigma$ be a root system in a Euclidean vector space $V$
	and $W$ the corresponding Weyl group.
	For $x\in V$ we denote by $C_x$ the convex hull of the $W$-orbit of $x$,
	\[
		C_x = \operatorname{conv} (Wx),
	\]
	and by $C_x^\pm$ its symmetrization
	\[
		C_x^\pm = \operatorname{conv} (Wx \cup W(-x)).
	\]
	More generally, we denote by $C^\pm$ the symmetrization of a compact set $C\subseteq V$,
	\[
		C^\pm = \operatorname{conv} (WC \cup W(-C)),
	\]
	and we note that $(C_x)^\pm = C_x^\pm$.
	Clearly, $C_x^\pm = C_x$ if and only if $-x \in Wx$.
	Let
	\[
		\mathcal{P}\coloneqq\{p\in \operatorname{Poly} (V_\C)^W \text{ elliptic}\}
	\]
	the set of elliptic $W$-invariant polynomials
	and
	\[
		\mathcal{P}_\R \coloneqq \{p\in \mathcal{P} \colon p(V)\subseteq\R\}
		=\{p\in \mathcal{P} \colon \overline{p(v)}=p(\overline{v}) \:\forall v\in V_\C\}
	\]
	the set of elliptic $W$-invariant polynomials with real coefficients.
	We will encounter the set $\mathcal{P}_\R$ for the self-adjoint spectral problem.
	The goal of this section is to identify the sets
	\[
		H(C) \coloneqq \{v\in V_\C \colon p(v) \in p(V + iC) \: \forall p\in \mathcal{P}\}
	\]
	and
	\[
		H_\R(C) \coloneqq \{v\in V_\C \colon p(v) \in p(V + iC)\cap \R \: \forall p\in \mathcal{P}_\R\}.
	\]
	We abbreviate
	\[
		H(x)=H([0,1]x) \quad \text{and} \quad H_\R(x)=H_\R([0,1]x)
	\]
	for $x\in V$.
	We start with the following well-known lemma.
	\begin{lemma}
		[{\cite[Lemma~III.3.11]{gaga}}]
		\label{la:Pseparatespoints}
		$\mathcal{P}$ separates the points of $V_\C/W$,
		i.e.~if $p(v)=p(w)$ for all $p\in \mathcal{P}$,
		then $v\in Ww \subseteq V_\C$.
	\end{lemma}
	\begin{proof}
		By \cite[Lemma~III.3.11]{gaga} we can find $p\in \operatorname{Poly} (V_\C)^W$
		that separates $Wv$ and $Ww$.
		As $Wv$ and $Ww$ are finite, we can add a small elliptic polynomial to $p$
		to get the same property for some $\tilde p\in \mathcal{P}$.
	\end{proof}

	\begin{proposition}
		\label{prop:imagecomplexpolynomial}
		For $x\in V$ we have
		\[
			p(V+i[0,1]x)=p(V+iC_x)
			\quad \forall p\in \mathcal{P}.
		\]
		In particular, $H(x) = H(C_x)$
		and $V+iC_x \subseteq H(x)$.
	\end{proposition}
	\begin{proof}
		Suppose $z \notin p(V + i[0,1]x)$.
		Since $p$ is elliptic,
		there is a connected open neighborhood $U$ of $[0,1]x$
		such that $z \notin p(V+i U)$.
		Therefore,
		$f\colon V+ iU \to \C, f(v)= (p(v)-z) ^{-1}$
		is an analytic function on $V+ iU$.
		The $W$-invariance of $p$ implies
		that $f$ is an analytic function on $\bigcup_{w\in W}  V + i wU$.
		By Bochner's tube theorem (see \cite[Thm.~2.5.10]{HormanderComplex}),
		$f$ extends to an analytic function to $V +i \operatorname{conv} (WU)$.
		Since $f$ is analytic, $f(\eta)$ is still given by $(p(\eta)-z)^{-1}$
		so that $z \notin p( V +i \operatorname{conv} (WU))$.
	\end{proof}

	For real polynomials we get the following analog including the symmetrization.

	\begin{proposition}
		\label{prop:lowerboundhullR}
		For $x\in V$ we have
		\begin{equation}
			\label{eq:imagerealpolynomial}
			p(V+i[0,1]x) \cap \R=p(V+iC_x^\pm)\cap \R
			\quad \forall p\in \mathcal{P}_\R.
		\end{equation}
		In particular, $H_\R(x) = H_\R(C_x^\pm)$.
	\end{proposition}
	\begin{proof}
		Suppose $y \in \R \setminus p(V +  i [0,1]x)$.
		Then $y \in \R \setminus p(V + i [-1,1]x)$.
		Indeed, $y = p(v-iw) \in \R$ for $w\in [0,1]x$ and $v\in V$
		would imply
		\[
			y = \overline{y} = \overline{p(v-iw)}
			= p(v+iw)\in p(V+[0,1]x)
		\]
		as $p\in \mathcal{P}_\R$
		which is a contradiction.
		Thus,
		\[
			p(V+i[0,1]x) \cap \R=p(V+i [-1,1])\cap \R.
		\]
		Now the same argument as for Proposition~\ref{prop:imagecomplexpolynomial}
		proves \eqref{eq:imagerealpolynomial}.
	\end{proof}

	In order to get a lower bound on $H_\R(x)$
	we
	need to analyze which elements of $V_\C$ have real values for $p \in \mathcal{P}_\R$.
	\begin{lemma}
		\label{la:hermitegeneral}
		Let
		\[
			V_{\operatorname{Her} } \coloneqq \{v\in V_\C \colon \overline{v} \in Wv\}.
		\]
		Then
		\[
			\{v\in V_\C \colon p(v)\in \R \: \forall p\in \mathcal{P}_\R\} = V_{\operatorname{Her}} .
		\]
	\end{lemma}
	\begin{remark}
		The notation $V_{\operatorname{Her}}$ stems from the fact that,
		for $V= \mathfrak{a}^\ast$ in a semisimple Lie group,
		the corresponding spherical function $\phi_{i\lambda}$
		for $\lambda\in \mathfrak{a}^\ast_{\operatorname{Her} }$
		is Hermitean,
		i.e.~can be used to define an inner product.
		For further reference see \cite{Cow23} and \cite[Ch.~IV.1]{gaga}.
	\end{remark}
	\begin{proof}
		By Chevalley's Theorem \cite[Thm.~3.5]{Hum90} there is a real set of generators
		for $\operatorname{Poly} (V_\C)^W$.
		By adding $p(v)=\langle v,v \rangle$ to these generators
		we get a set of generators of $\operatorname{Poly} (V_\C)^W$
		in $\mathcal{P}_\R$.
		Thus, for $v\in V_\C$, $p(v)\in \R$ for all $p\in \mathcal{P}_\R$
		is equivalent to $p(v)=p(\overline{v})$ for all $p\in \mathcal{P}_\R$
		and therefore also for all $p\in \operatorname{Poly} (V_\C)^W$.
		Lemma~\ref{la:Pseparatespoints} implies $\overline{v} \in Wv$.
	\end{proof}

	\begin{corollary}
		\label{cor:hullhermit}
		For $x\in V$ we have
		\[
			V_{\operatorname{Her}} \cap (V + iC_x^\pm) \subseteq H_\R(x) \subseteq V_{\operatorname{Her} }.
		\]
	\end{corollary}

	We now turn to an upper bound for $H_\R(x)$.
	We need the following lemma.

	\begin{lemma}
		\label{la:sepellipse}
		Let $V$ be a Euclidean vector space,
		$C\subset V$ be compact and convex and let
		\[
			C^{\pm}\coloneqq\operatorname{conv}(C\cup(-C)).
		\]
		If $y\notin C^{\pm}$, then there exists a positive definite symmetric operator $A:V\to V$ such that
		\[
			\langle Ac,c\rangle<\langle Ay,y\rangle
			\qquad\text{for all } c\in C.
		\]
	\end{lemma}
	Usually, one separates convex sets by hypersurfaces,
	i.e.~by a level set of a linear form.
	To ensure ellipticity of our polynomials we use ellipsoids
	which are compact defined by
	the positive definite operator $A$
	to separate points from compact convex sets.
	Since ellipsoids are symmetric,
	we have to assume that $y\notin C^\pm$ instead of just $y\notin C$.
	See Figure~\ref{fig:ellipse} for a visualization.
	\begin{figure}[h]
		\centering
		\includegraphics[width=\textwidth, trim = 1cm 1.5cm 5cm 4cm, clip]{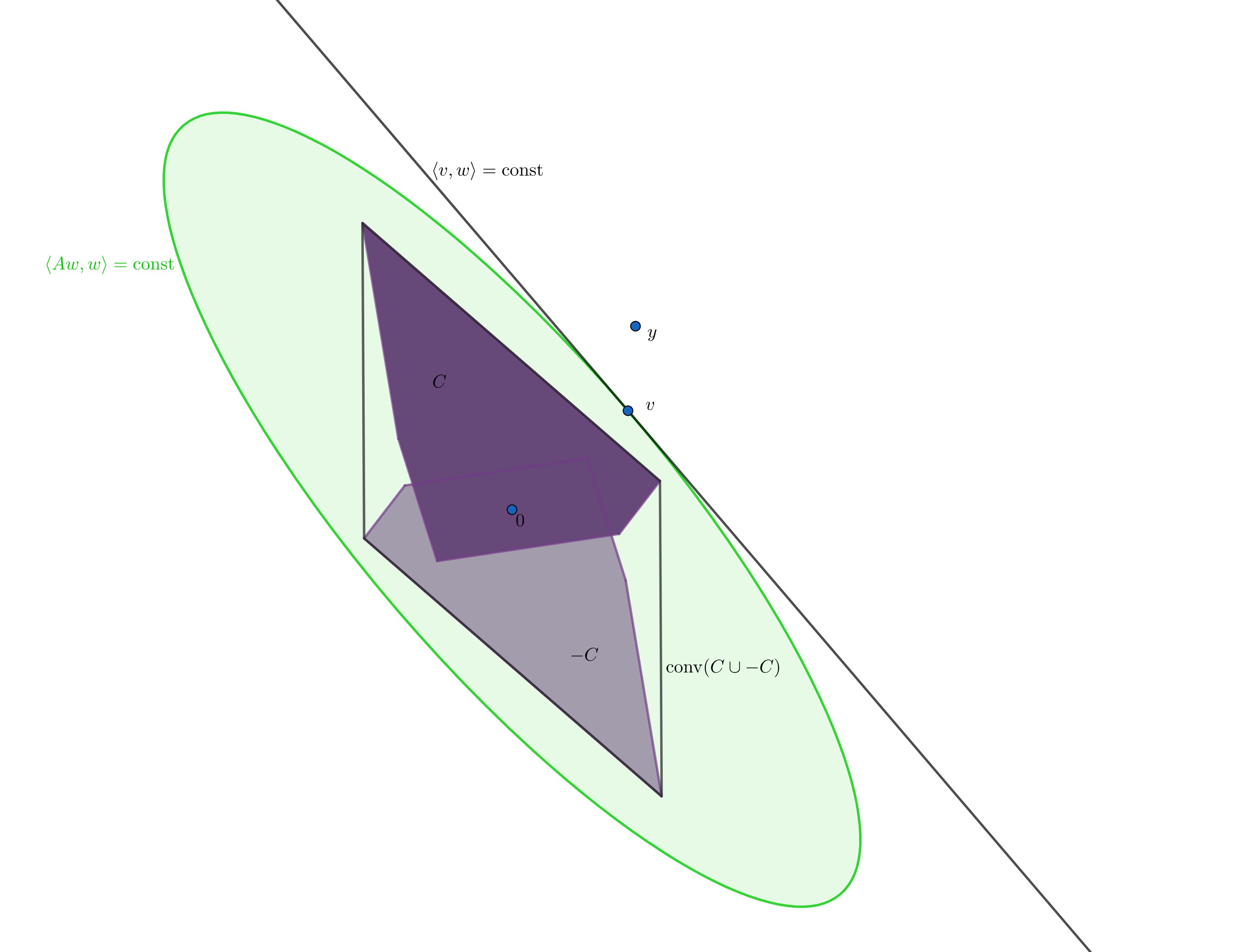}
		\caption{Comparison of separation by hypersurface and by ellipsoid}
		\label{fig:ellipse}
	\end{figure}

	\begin{proof}
		Since $C^{\pm}$ is compact, convex, and $-C^\pm = C^\pm$,
		we can find $v\in V$ such that
		\[
			b\coloneqq \langle v, y\rangle > \max_{c\in C^{\pm}}\langle v,c\rangle
			=\max_{c\in C}|\langle v,c\rangle| \eqqcolon M.
		\]
		Let $R=\max_{c\in K}\|c\|$.
		Since $b^2 -M^2 >0$ we can choose $\varepsilon>0$ small such that
		\[
			\varepsilon (R^2-\|y\|^2) < b^2 -M^2.
		\]
		For $w\in V$ define
		\[
			Aw = \langle v,w\rangle v + \varepsilon w.
		\]
		By definition
		\[
			\langle Aw,w'\rangle = \langle v,w\rangle \langle v,w'\rangle + \varepsilon \langle w,w'\rangle,
			\qquad w,w'\in V,
		\]
		and
		\[
			\langle Aw,w\rangle = \langle v,w\rangle^ 2 + \varepsilon \|w\|^2 >0,
			\qquad w\in V,
		\]
		so that $A$ is symmetric positive definite.
		For $c\in C$ we have
		\[
			\langle A c,c\rangle \leq M^2+\varepsilon R^2
			< b^2 +\varepsilon \|y\|^2= \langle A y,y\rangle
		\]
		which proves the lemma.
	\end{proof}

	\begin{proposition}
		\label{prop:sepinPR}
		Let $C\subseteq V$ be compact, convex, and $W$-invariant.
		Let $z = x+iy\in V_\C$ with $y=\Im z \notin C^{\pm}$.
		Then there is $p\in \mathcal{P}_\R$ such that
		\[
			p(z)=0
			\qquad
			\text{and}
			\qquad
			0\notin p(V + iC).
		\]
	\end{proposition}

	\begin{proof}
		Choose $A \colon V\to V$
		as in Lemma~\ref{la:sepellipse} for $y=\Im z$.
		Let
		\[
			B\coloneqq \langle Ay,y\rangle>0
		\]
		and define the quadratic polynomial
		\[
			q(v)=\langle A(v -x),v-x \rangle+B,
			\qquad v\in V_\C.
		\]
		We note that $A$ and $\langle\cdot,\cdot\rangle$ are extended complex (bi-)linearly.
		Then
		\[
			q(z)= -\langle Ay, y\rangle+B=0
		\]
		and
		\[
			q(v)\in \R
		\]
		for all $v\in V$.
		The leading homogeneous part of $q$ is $\langle Av,v\rangle$
		which implies by positive definiteness of $A$
		that $q$ is elliptic.
		We now set
		\[
			p(v)=\prod_{w\in W}q(wv).
		\]
		$p$ is clearly elliptic, $W$-invariant and satisfies $p(v)\in \R$ for $v\in V$.
		Thus $p\in \mathcal{P}_\R$
		and $p(z)=0$.

		It remains to show that $p$ has no zero in $V+iC$.
		For a fixed $w\in W$ and $v\in V+iC$ we have
		\begin{align*}
			q(w v)
  &=\langle A(w v -x),wv-x \rangle+B\\
  &=\langle A(w \Re v -x), (w\Re v-x)\rangle - \langle Aw \Im v , w\Im v\rangle +B + 2i \langle A(w \Re v -x),w\Im v \rangle.
		\end{align*}
		It follows
		\[
			\Re q(w v)
			\geq  
			B- 	\langle A w\Im v , w \Im v\rangle
		\]
		Because $C$ is $W$-invariant, $w \Im v \in C$.
		Hence Lemma~\ref{la:sepellipse} gives
		\[
			B> \langle A w\Im v , w \Im v\rangle
		\]
		and therefore
		\[
			\Re q(w v)>0.
		\]
		In particular, $q(w v) \neq 0$ and therefore $p(v)\neq 0$.
		This completes the proof.
	\end{proof}
	We use the above proposition to determine $H_\R(x)$.
	To this end, let us choose a positive (closed) Weyl chamber $V_+$ in $V$.
	Then there exists a unique element $w_0\in W$ such that $w_0 V_+ \subseteq -V_+$
	and $w_0^2=1$.
	Thus $\iota = -w_0$ is an involution of $V$ fixing $V_+$.
	We denote by $x^\circ$ the element $\frac 12 (x + \iota x)$.
	\begin{theorem}
		\label{thm:realpolynomialhull}
		For $x\in V$ we have
		\[
			H_\R(x)= V_{\operatorname{Her}} \cap (V+i C_x^\pm)
			= V _{\operatorname{Her} } \cap (V + iC _{x^\circ})
		\]
		and
		\[
			V + i C_x \subseteq H(x) \subseteq V+iC_x^\pm.
		\]
	\end{theorem}
	\begin{proof}
		Proposition~\ref{prop:sepinPR} shows $H_\R(C_x^\pm)\subseteq V+i C_x^\pm$.
		The theorem then follows from Proposition~\ref{prop:lowerboundhullR}
		and Corollary~\ref{cor:hullhermit} and Lemma~\ref{la:nonhermitianconvexhull} below.
	\end{proof}

	\begin{remark}
		It follows that $H(x)=V+iC_x$
		if $-x\in Wx$.
	\end{remark}

	\begin{lemma}
		\label{la:nonhermitianconvexhull}
		\[
			V_{\operatorname{Her}} \cap (V+i C_x^\pm)
			= V_{\operatorname{Her} }\cap ( V + iC _{x^\circ})
		\]
	\end{lemma}
	\begin{proof}
		Since
		\[
			x^\circ=\frac{1}{2}x + \frac 12 w_0(-x)\in C_x^\pm
		\]
		it is clear that the right hand side is contained in the left hand side.

		Conversely, let $u' + iu \in V_{\operatorname{Her} }$ with $u\in C_x^\pm$.
		Without loss of generality we assume that $u\in V_+$.
		\cite[Lemma~IV.8.3]{gaga} implies that it is sufficient to show
		\[
			\langle u,v \rangle \leq \langle x^\circ,v\rangle
			\quad \forall v\in V_+.
		\]
		By definition of $V_{\operatorname{Her} }$ there is $w\in W$
		such that $wu=-u \in -V_+$.
		Similarly, $w_0u \in -V_+$.
		This implies $w_0u=-u$ (see e.g.~\cite[Thm.~1.12 (b)]{Hum90}).
		We write
		\[
			u=\sum_{\pm,w\in W}^{} c_w^\pm (\pm wx)
			,
			\quad 
			\sum_{\pm,w\in W}^{} c_w^\pm =1, c_w^\pm \in [0,1]
			.
		\]
		Then
		\begin{align*}
			u&= \frac{u-w_0 u}{2} = 
			\frac 12 \sum_{\pm,w\in W}^{} c_w^\pm (\pm wx)
			+ \frac 12 \sum_{\pm,w\in W}^{} c_w^\pm (\mp w_0wx)
			\\
			 &=
			 \frac 12 \sum_{w\in W}^{} c_w^+ (wx - w_0wx)
			 - \frac 12 \sum_{w\in W}^{} c_w^- (wx- w_0wx)
			 \\
			 &=
			 \frac 12 \sum_{w\in W}^{} c_w^+ (wx - w_0wx)
			 + \frac 12 \sum_{w\in W}^{} c_{w_0w}^- (wx- w_0wx)
			 \\
			 &=\sum_{w\in W} (c_w^++c_{w_0w}^-) \frac{wx - w_0wx}2.
		\end{align*}
		Since
		\[
			\sum_{w\in W}^{} (c_w^+ + c_{w_0 w}^-) = 1
		\]
		the proof is complete if we can show
		\[
			\left\langle \frac{wx -w_0 wx}2,v \right\rangle \leq \langle x^\circ,v\rangle.
		\]
		We use once more \cite[Lemma~IV.8.3]{gaga}
		to see that
		\begin{equation}
			\label{eq:wxbdd}
			\langle wx,v\rangle \leq \langle x,v\rangle
			\quad \forall v\in V_+.
		\end{equation}
		Now we can calculate
		\begin{align*}
			\langle wx-w_0 wx,v \rangle  = 
			\langle wx,v \rangle  + \langle wx,-w_0 v\rangle \leq
			\langle x,v \rangle  + \langle x,-w_0 v\rangle =
			2\langle x^\circ ,v\rangle
		\end{align*}
		where we used \eqref{eq:wxbdd} for both $v$ and $-w_0v$.
		This completes the proof.
	\end{proof}

	\section{Spectral theory on \texorpdfstring{$\Gamma \backslash G/K$}{Gamma\\G/K}}%
	\label{sec:Spectralonlocally}

	\subsection{Spectrum of individual self-adjoint operators}
	\label{sub:specindividual}
	In this section we analyze the spectrum of
	self-adjoint elliptic operators $D\in \mathbb D(G/K)$
	on locally symmetric spaces.
	To this end we fix an arbitrary discrete subgroup $\Gamma \subseteq G$.
	Let us recall the definition of the limit cone $\mathcal{L}_\Gamma$
	\cite{Ben96}:
	\begin{align*}
		\mathcal{L}_\Gamma = 
		\{v\in \mathfrak{a} \colon
			|\mu_+(\Gamma) \cap \mathcal{C}| = \infty
			\text{ for all open cones } \mathcal{C} \subseteq \mathfrak{a}_+
		\text{ containing } v\}.
	\end{align*}
	In order to utilize the resolvent estimate Proposition~\ref{prop:decaykernel}
	we assume in the present Section~\ref{sec:Spectralonlocally}
	that
	\[
		\mathcal{L}_\Gamma \subseteq \mathfrak{a}_{++}.
	\]
	We also recall the definition of the modified critical exponent
	$\delta_\lambda'$ for $\lambda$ positive on $\mathcal{L}_\Gamma$:
	\begin{equation}
		\label{eq:defcritexp}
		\delta'_\lambda = \inf\left\{s\in \R \colon \sum_{\gamma\in \Gamma} e^{(-s\lambda -\rho)\mu_+(\gamma)}<\infty\right\}.
	\end{equation}
	We note that $\delta'_\lambda$ can even be defined for every $\lambda$
	which is non-negative on $\mathcal{L}_\Gamma$ \cite{limitcone}.
	Since we only work with $\lambda\in \mathfrak{a}_+^\ast$
	and we assumed $\mathcal{L}_\Gamma \subseteq \mathfrak{a}_{++}$
	the definition \eqref{eq:defcritexp} is enough for our purpose.

	Let us now fix $p\in \mathcal{P}_\R$.
	In this section we will connect the spectrum of $D_{p}=\operatorname{HC} ^{-1}(p(i\cdot))$
	with $\delta_\lambda'$.

	We first note that $D_p$ is an elliptic essentially self-adjoint differential operator
	on $L^2(G/K)$.
	Indeed, the cotangent bundle $T^\ast G/K = G \times_K \mathfrak{p}^\ast$
	is the associated vector bundle for the coadjoint representation
	of $K$ on $\mathfrak{p}^\ast$.
	The principal symbol of $D_p$ is given by
	\[
		G \times _K \mathfrak{p}^\ast \ni [g,\xi]\mapsto p(i \operatorname{Ad}^\ast(k) \xi)
	\]
	where $\Ad^\ast(k) \xi$ is the unique element in $\Ad^\ast(K)\xi \cap \mathfrak{a}^\ast$.
	Thus $D_p$ is elliptic.

	$D_p$ is symmetric on $L^2(G/K)$
	as
	\begin{align}
		\label{eq:symmetricDO}
		\chi_\lambda(D^\ast) = \overline{\chi_{-\overline{\lambda}}(D)}
		=\overline{p(-i \overline{\lambda})}=
		p(i\lambda)=\chi_\lambda(D),
		\quad
		\lambda\in \mathfrak{a}^\ast_\C,
	\end{align}
	since $p\in \mathcal{P}_\R$
	and therefore $D_p =D_p^\ast$.
	The essential self-adjointness now follows from
	\cite[Thm.~2.2]{NS59}.

	Let us denote by ${}_\Gamma D_p$ the differential operator induced by $D_p$
	on $C_c(\Gamma \backslash G/K)$.
	The same arguments as above show that ${}_\Gamma D_p$ is elliptic 
	and essentially self-adjoint on $L^2(\Gamma \backslash G/K)$.
	We denote by $\sigma(_\Gamma D_p)\subseteq \R$ its spectrum.

	The main result of this section is the following theorem.
	\begin{theorem}
		\label{thm:spectrumsinglesa}
		Let $p\in \mathcal{P}_\R$ with $\deg p >\frac 12 \dim N +\dim A-1$
		and $D = D_{p}$ as above.
		Then
		\[
			\sigma({}_\Gamma D_p) \subseteq
			p(\mathfrak{a}^\ast + i [0,\max(0,\delta'_\lambda)]\lambda)
		\]
		for every $\lambda\in \mathfrak{a}_+^\ast$.
	\end{theorem}
	We first note that $\sigma(D_p)$ equals $\overline{p(i \mathfrak{a}^\ast)} \subseteq \R$
	(see \cite[Lemma~3.2]{WW23}) and thus is of the form $\pm [a,\infty)$
	for $a\in \R$ as $p$ is elliptic.
	Without loss of generality, $\sigma(D_p)=[a,\infty)$.
	We use the following lemma to push the convergence
	of the averaged resolvent kernels to a proper $L^2$-resolvent.

	\begin{lemma}
		\label{la:specfromestimate}
		Suppose $|k(D_p,z,g^{-1} \gamma h)| \leq f(\gamma)$ locally uniform in $z<a$
		and for $g,h$ in compact sets of $G$
		(except for possibly finitely many $\gamma$ where $g ^{-1} \gamma h \in K$).
		If $\sum_{\gamma\in \Gamma} f(\gamma) < \infty$,
		then $z\notin \sigma(_\Gamma D_p)$.
	\end{lemma}
	\begin{proof}
		The kernel of the resolvent $(_\Gamma D_p-z)^{-1}$ if it exists is given by $\sum_{\gamma\in \Gamma} k(D_p,z,gK,\gamma hK)$.
		However, the convergence of this series does not immediately imply the well-definition of the resolvent
		as an operator on $L^2(\Gamma \bk G/K)$.
		Therefore, we use Stone's formula as in \cite[Prop.~4.4]{WW23}.

		Let $P_I$ be the spectral projector of ${}_\Gamma D_p$ for a Borel subset $I\subseteq\R$.
		We have to show $P_{[z-\delta,z+\delta]}=0$ for some $\delta >0$.
		By Stone's formula \cite[Prop.~5.14]{Sch12}
		\begin{align*}
			\frac 12 (P_{[a,b]}+P_{]a,b[}) = \lim_{\varepsilon\to 0} \frac 1 {2\pi i} \int_a^b (_\Gamma D_p-(t+i\varepsilon))^{-1} - (_\Gamma D_p-(t-i\varepsilon))^{-1}\: dt.
		\end{align*}
		Here the limit $\varepsilon \to 0$ is understood in the strong sense.
		The occurring resolvents are well-defined as $D_p$ is symmetric.
		Thus we have to show 
		\begin{align*}
			\lim_{\varepsilon\to 0}  \int_{z-\delta}^{z + \delta} 
			(_\Gamma D_p-(t+i\varepsilon))^{-1} - (_\Gamma D_p-(t-i\varepsilon))^{-1} \: dt = 0.
		\end{align*}
		Let $\varphi,\psi \in C_c^\infty(\Gamma \bk G/K)$
		and $\tilde \varphi$ (resp.~$\tilde \psi$) in $C_c^\infty(G/K)$
		such that $\sum_{\gamma\in \Gamma} \tilde \varphi (\gamma gK)=\varphi(\Gamma gK)$ (resp.~
		$\sum_{\gamma\in \Gamma} \tilde \psi (\gamma gK)=\psi(\Gamma gK)$).
		Then we must show
		\begin{align*}
			\lim_{\varepsilon\to 0}  \int_{z-\delta}^{z + \delta}
			\langle
			[(_\Gamma D_p-(t+i\varepsilon))^{-1} - (_\Gamma D_p-(t-i\varepsilon))^{-1}]
			\varphi,\overline\psi\rangle_{L^2(\Gamma\bk G/K)} \: dt = 0.
		\end{align*}
		By \cite[Lemma~4.2]{WW23} we thus have to show
		\begin{align*}
			\lim_{\varepsilon\to 0}  \int_{z-\delta}^{z + \delta}
			\sum_{\gamma\in \Gamma}
			\int_{G/K}^{} 
			\int_{G/K}^{} 
			[k(D_p,t+i\varepsilon ,gK,\gamma hK) - 
			k(D_p,t-i\varepsilon ,gK,\gamma hK)]
			\tilde\varphi(  hK) \tilde \psi(gK)\: d{gK} \: d{hK} 
			\: dt = 0.
		\end{align*}
		We can estimate the integrand by
		\begin{align*}
			2f(\gamma)
			|\tilde\varphi|(  hK)| \tilde \psi|(gK)
		\end{align*}
		which is integrable over $[z-\delta,z+\delta]\times \Gamma \times G/K \times G/K$.
		We can therefore use Lebesgue's theorem to move $\lim_{\varepsilon\to 0}$ under the integral sign.
		This finishes the proof as the difference of the kernels converges to $0$.
		The finitely many $\gamma\in \Gamma$ where the estimate does not hold also converge to $0$ as
		they simply give
		\begin{align*}
			\lim_{\varepsilon\to 0}&  \int_{z-\delta}^{z + \delta}
			\int_{G/K}^{} 
			\int_{G/K}^{} 
			[k(D_p,t+i\varepsilon ,gK,\gamma hK) - 
			k(D_p,t-i\varepsilon ,gK,\gamma hK)]
			\tilde\varphi(  hK) \tilde \psi(gK)\: d{gK} \: d{hK} 
			\: dt
			\\
					       &=\lim_{\varepsilon\to 0}  \int_{z-\delta}^{z + \delta}
					       \langle
					       [( D_p-(t+i\varepsilon))^{-1} - (D_p-(t-i\varepsilon))^{-1}]
					       \varphi(\gamma ^{-1} \cdot),\overline\psi\rangle_{L^2(\Gamma\bk G/K)} \: dt
					       \\
					       &=\frac{2\pi i}{2} \langle(P'_{[z-\delta,z+\delta]} + P'_{(z-\delta,z+\delta)}) \tilde \varphi(\gamma ^{-1} \cdot), \overline{\tilde \psi}\rangle
					       =0
		\end{align*}
		as $z\notin \sigma(D)$.
		Here $P'$ is the spectral projector for $D_p$.
	\end{proof}

	\begin{proof}
		[Proof of Theorem~\ref{thm:spectrumsinglesa}]
		Let $\lambda\in \mathfrak{a}_+^\ast$ and $z<a$ as before.
		We define
		\[
			\eta_\lambda(z) \coloneqq
			\inf \{ t\geq 0 \colon z= \chi_\mu(D_p) \text{ for } \Re \mu = t\lambda\}
		\]
		It follows from the ellipticity of $p$
		and $z\notin [a,\infty) = \{\chi_\mu(D_p)\colon \mu\in i\mathfrak{a}^\ast\}$
		that $\eta_\lambda(z)>0$.
		Hence, for any $\varepsilon>0$, we can apply Proposition~\ref{prop:decaykernel} to $(\eta_\lambda(z) -\varepsilon)\lambda$.
		We obtain
		\[
			|k(D_p,z,e^H)| \leq C e^{(- (\eta_\lambda(z)-\varepsilon)\lambda - \rho)H}.
		\]

		If $\eta_\lambda(z)> \delta'_\lambda$, then we choose $s\in (\max(0,\delta'_\lambda),\eta_\lambda(z))$
		so that 
		\[
			\sum_{\gamma\in \Gamma} e^{(-s\lambda -\rho)\mu_+(\gamma)}<\infty
		\]
		and
		\begin{equation}
			\label{eq:estkernelbypoincare}
			|k(D_p,z,e^H)|\leq C e^{(-s\lambda-\rho)H}.
		\end{equation}
		From Lemma~\ref{la:specfromestimate}
		we now get $z\notin \sigma(_\Gamma D_p)$.
		We note that we need $\mathcal{L}_\Gamma \subseteq \mathfrak{a}_{++}$ in order to get
		the estimate \eqref{eq:estkernelbypoincare}
		for almost all $\gamma\in \Gamma$
		because of the assumption on $H$ in Proposition~\ref{prop:decaykernel}.
		Clearly, $\eta_\lambda(z)> \delta'_\lambda$ if and only if $\eta_\lambda(z) > \max(0,\delta'_\lambda)$
		if and only if
		$z\notin \chi_{[0,\max(0,\delta'_\lambda)]\lambda +i \mathfrak{a}^\ast}(D_p)$.
		Thus we have
		\[
			z\notin \chi_{[0,\max(0,\delta'_\lambda)]\lambda +i \mathfrak{a}^\ast}(D_p)
			\implies
			z\notin \sigma(_\Gamma D_p)
		\]
		i.e.
		\begin{equation*}
			\label{eq:inclusionspec}
			\sigma(_\Gamma D_p) \subseteq 
			\chi_{[0,\max(0,\delta'_\lambda)]\lambda +i \mathfrak{a}^\ast}(D_p)
			\cap \R
			=p(\mathfrak{a}^\ast + i [0,\max(0,\delta_\lambda')]\lambda)\cap \R.
			\qedhere
		\end{equation*}
	\end{proof}

	\begin{remark}
		We note that Theorem~\ref{thm:spectrumsinglesa}
		is not sharp
		for all $\lambda$
		in the sense that
		\[
			\min \sigma(_\Gamma D_p) > 
			\min p(\mathfrak{a}^\ast +i[0,\max(0,\delta'_\lambda)]\lambda )\cap \R
		\]
		might hold.
		(One always has $\geq$ by Theorem~\ref{thm:spectrumsinglesa}.)
		Indeed, ignoring the assumption on the degree
		and considering the shifted Laplacian $D = \Delta-\|\rho\|^2$,
		we have $\chi_\lambda(D)=-\langle\lambda,\lambda\rangle$
		and $p(\lambda)= \langle\lambda,\lambda\rangle$.
		One easily calculates that
		\[
			\min p(\mathfrak{a}^\ast +i[0,\max(0,\delta'_\lambda)]\lambda )\cap \R
			= -\max(0,\delta'_\lambda)^2 \|\lambda\|^2
		\]
		On the other hand, it is known that
		$\min \sigma(_\Gamma D)=- \max(0,\delta')^2$
		where
		\[
			\delta' = \inf \left\{ s\in \R\colon \sum_{\gamma\in \Gamma} e^{-s\|\mu_+(\gamma)\|-\rho(\mu_+(\gamma))} \infty\right\}
			.
		\]
		Hence, if $\delta' >0$,
		Theorem~\ref{thm:spectrumsinglesa} yields
		$\delta' \leq \delta'_\lambda \|\lambda\|$ for all $\lambda\in \mathfrak{a}_+^\ast$.
		By \cite[Lemma~3.1]{limitcone}
		$\lambda\mapsto \delta'_\lambda \|\lambda\|$ is strictly convex
		which implies that we can have equality for at most one $\lambda$.
		If $\Gamma$ is Zariski dense, then this minimum is indeed achieved for $\lambda=\langle v_\Gamma',\cdot\rangle$
		where $v_\Gamma'$ maximizes $\psi_\Gamma-\rho$ on $\|v\|=1$
		where $\psi_\Gamma$ is the growth indicator function.
		See \cite{limitcone} for more details.
	\end{remark}

	\subsection{The joint spectrum}%
	\label{sub:The joint spectrum}
	In this section we consider the joint spectrum of all operators in
	$\mathbb{D}(G/K)$
	instead of their spectra individually.
	There are many equivalent definition for the joint spectrum, see \cite[Prop.~3.6]{WW23}.
	We will give a definition using only self-adjoint operators.
	This has the advantage that we do not have to use any decomposition
	of $L^2(\Gamma\bk G)$ into irreducible representations.
	\begin{definition}
		Let $p_1,\ldots,p_r$ be homogeneous generating polynomials
		of $\operatorname{Poly}(\mathfrak{a}_\C^\ast)^W$
		with $D_{p_j}$ symmetric
		(equivalently $p_j( \mathfrak{a}^\ast)\in\R$
		or $\overline{p_j(\lambda)} = p_j( \overline{\lambda})$).
		The set
		\[
			\left\{\lambda\in \mathfrak{a}_\C^\ast
				\colon {}_\Gamma \sum_{j=1}^r (D_{p_j} - p_j(\lambda))^\ast (D_{p_j}-p_j(\lambda))
				\text{ is not invertible on }
			L^2(\Gamma \backslash G/K)\right\}
		\]
		is called the joint spectrum $\wt \sigma(\Gamma \bk G/K)$ of $\Gamma\bk G/K$.
		We also write $\wt \sigma$ for $\wt \sigma(\Gamma\bk G/K)$.
	\end{definition}
	Note that the operator in question is essentially self-adjoint
	by \cite[Cor.~2.4]{NS59}
	so $\wt \sigma$ is well-defined.

	We will use the following description of $\wt \sigma$
	reducing the question to elliptic operators.
	\begin{lemma}
		\label{la:defjointspectrumelliptic}
		The joint spectrum $\wt \sigma$ coincides with
		\[
			\{\lambda\in \mathfrak{a}_\C^\ast\colon
				p(i\lambda)\in \sigma(_\Gamma D_{p}) \:
			\forall p\in \mathcal{P}_\R\}.
		\]
	\end{lemma}
	\begin{proof}
		Let us call the set in question $S$.
		By \cite[Prop.~3.6]{WW23}, $S$ is contains $\wt \sigma$.
		For the converse direction consider
		\[
			p(\mu)=\sum_{j=0}^{r} (p_j(\mu) - \overline{p_j(i\lambda)})
			(p_j(\mu) -p_j(i\lambda))
		\]
		where $p_0(\mu)=-\langle \mu,\mu\rangle^N$
		for $N$ large so that $p$ is elliptic.
		Since $p( \mathfrak{a}^\ast)\subseteq\R$, we have $p\in \mathcal{P}_\R$.
		Furthermore, the operator $D_{p}$ is
		precisely $ \sum_{j=0}^r (D_{p_j} - p_j(i\lambda))^\ast (D_{p_j}-p_j(i\lambda))$.
		Hence, if $\lambda\in S$, then
		\[
			0=p(i\lambda) \in
			\sigma(_\Gamma D_{p}).
		\]
		But $(D_{p_0} - p_0(i\lambda))^\ast (D_{p_0}-p_0(i\lambda))$
		is a positive operator
		so that
		\[
			0\in \sigma\left({}_\Gamma \sum_{j=1}^r (D_{p_j} - p_j(i\lambda))^\ast (D_{p_j}-p_j(i\lambda))
			\right)
		\]
		and therefore $\lambda\in \wt \sigma$.
	\end{proof}

	Since the operators $D_{p}$, $p\in \mathcal{P}_\R$, are symmetric and therefore have real 
	spectrum,
	we directly get the well-known Hermitian property
	$-\overline{\lambda} \in W\lambda$
	of the joint spectrum from Lemma~\ref{la:hermitegeneral}.

	We now use Theorem~\ref{thm:spectrumsinglesa} to bound the joint spectrum.

	\begin{proposition}
		\label{prop:jointspecinintersectionofpolys}
		For $\lambda\in \mathfrak{a}^\ast_+$ we have
		\[
			\wt \sigma \subseteq
			\{\lambda\in \mathfrak{a}_\C^\ast
				\colon
			p(i\lambda) \in p(\mathfrak{a}^\ast + i[0,\max (0,\delta'_\lambda)]\lambda) \: \forall p\in \mathcal{P}_\R\}.
		\]
	\end{proposition}
	\begin{proof}
		It follows immediately from Theorem~\ref{thm:spectrumsinglesa}
		and Lemma~\ref{la:defjointspectrumelliptic}
		that
		\[
			p(i\lambda) \in p(\mathfrak{a}^\ast + i[0,\max (0,\delta'_\lambda)]\lambda)
		\]
		for $\lambda\in \wt \sigma$
		if $\deg p$ big enough.
		For $p$ without sufficiently big degree
		we use $(p-p(i\lambda))^N$, $N\gg 0$, to complete the proof.
	\end{proof}

	Our work in Section~\ref{sec:Separating tubes by invariant polynomials}
	now shows \eqref{eq:bdspecnonHerm}.
	In fact,
	Proposition~\ref{prop:jointspecinintersectionofpolys}
	implies
	\[
		-i \wt \sigma \subseteq H_\R(\max(0,\delta'_\lambda) \lambda)
	\]
	and from Theorem~\ref{thm:realpolynomialhull} together with
	Lemma~\ref{la:nonhermitianconvexhull} it follows
	\[
		\Im H_\R(
		\max(0,\delta'_\lambda) \lambda)
		\subseteq \max(0,\delta'_\lambda)\operatorname{conv}\left(W \frac{\lambda-w_0 \lambda}{2}\right)
	\]
	completing the proof.

	\section{Convergence Domain of the Poincaré series}%
	\label{sec:prooftempered}
	In this section 
	we study the region of convergence of the Poincaré series
	$P_\lambda = \sum_{\gamma\in \Gamma} e^{-\lambda(\mu_+(\gamma))}$.
	This will lead to a new proof for Theorem~\ref{thmB}
	from Theorem~\ref{thmA}
	which we think is clearer than the one given in \cite{LWW},
	see Section~\ref{sub:Precise bounds on the domain of convergence}.
	We will also complete the proof of Theorem~\ref{thmA} in Section~\ref{sub:Completion of the proof for thmA}.

	Recall that we denote the longest Weyl group element associated to
	$\mathfrak{a}_+$ by $w_0$.
	In this section we write $\iota$ for the involution $-w_0$
	on $\mathfrak{a}^\ast$ and $\mathfrak{a}$.
	Moreover, we will denote by 
	$\mathfrak{a}_\iota^\ast$ the subspace $\ker(\iota - 1)  = \{\lambda\in \mathfrak{a}^\ast\colon \iota \lambda=\lambda\}$ and by
	$\mathfrak{a}^\ast_{+,\iota}$ the subcone $\mathfrak{a}^\ast_+ \cap \mathfrak{a}^\ast_\iota$.

	\subsection{Elementary properties}%
	\label{sub:Elementary properties}

	As before, let $\mathcal{L}_\Gamma$ be the limit cone
	of a discrete subgroup $\Gamma \subseteq G$.
	Contrary to Section~\ref{sec:Spectralonlocally}
	we do not have to assume that $\mathcal{L}_\Gamma \subseteq \mathfrak{a}_{++}$
	due to the correctness of Theorem~\ref{thmA}
	without this assumption.
	However, to deduce Theorem~\ref{thmB}
	we will assume that $\mathcal{L}_\Gamma$ is disjoint from some walls of $\mathfrak{a}_{++}$ as in \cite{limitcone}.

	We denote by $\mathcal{L}_\Gamma^\star$ the dual cone
	\[
		\mathcal{L}_\Gamma^\star \coloneqq \{\lambda\in \mathfrak{a}^\ast\colon \lambda(\mathcal{L}_\Gamma) \geq 0\}.
	\]
	and by
	\[
		\operatorname{int}\mathcal{L}_\Gamma^\star=
		\{\lambda\in \mathcal{L}_\Gamma \colon
		\lambda(\mathcal{L}_\Gamma \setminus \{0\})>0\}
	\]
	its interior.
	Let
	\[
		R\coloneqq \operatorname{cl} \left\{ \lambda\in \mathcal{L}_\Gamma^\star \colon \sum_{\gamma\in \Gamma} e^{-\lambda(\mu_+(\gamma))} <\infty\right\}.
	\]
	Clearly, $2\rho\in R$ as $\Gamma$ is discrete.
	The property $\wt \sigma \subseteq i \mathfrak{a}^\ast$
	is equivalent to asking whether $\rho\in R$.
	Due to the $\rho$-shift present in our setting,
	we consider
	\[
		R' \coloneqq R-\rho =\operatorname{cl} \left\{ \lambda\in \mathcal{L}_\Gamma^\star \colon \sum_{\gamma\in \Gamma} e^{-(\lambda+\rho)(\mu_+(\gamma))} <\infty\right\}
	\]
	so that $\rho\in R$ if and only if $0\in R'$.
	We formulate the following classical lemma that follows from the Hölder inequality.
	\begin{lemma}
		\label{la:Rconvex}
		$R$ and $R'$ are convex.
	\end{lemma}

	From the definition of $\delta'_\lambda$ it is clear that
	\[
		R' = \bigcup_{\lambda\in \mathcal{L}_\Gamma^\star} [\delta'_\lambda,\infty)\lambda.
	\]
	It is also obvious that, if $\mu\in R'$ and $\lambda\geq \mu$ on $\mathcal{L}_\Gamma$,
	then $\lambda\in R'$.
	In particular,
	for $\mu\in R'$,
	\begin{equation}
		\label{eq:addlimitconeinR}
		\mu + {}_+\mathfrak{a}^\ast \subseteq
		\mu + \mathcal{L}_\Gamma^\star\subseteq
		R'
	\end{equation}
	where
	\[
		{}_+\mathfrak{a}^\ast = \{\mu\in \mathfrak{a}^\ast\colon \mu \geq 0 \text{ on } \mathfrak{a}_+\}=
		\sum_{\alpha\in \Pi} \R_{\geq 0} \alpha.
	\]
	\begin{remark}
		Let $\psi_\Gamma$ be the growth indicator function (see e.g.~\cite[Eq.~(1.1)]{KMO24}).
		For $\lambda \in \operatorname{int}\mathcal{L}_\Gamma^\star$ \cite[Lemma~III.3.1]{Qui02}
		or if $\psi_\Gamma$ is concave \cite[Thm.~2.5]{KMO24}
		one has
		\[
			\lambda\in R \iff \psi_\Gamma \leq \lambda.
		\]
		Thus $R$ coincides with
		\begin{equation}
			\label{eq:defRbygif}
			\{\lambda\in \mathcal{L}_\Gamma^\star \colon \psi_\Gamma \leq \lambda\}
		\end{equation}
		in $\operatorname{int} \mathcal{L}_\Gamma^\star$.
		To avoid some technicalities in \cite{limitcone}
		it seems to be better to work with \eqref{eq:defRbygif} as the definition for $R$.
		Similarly, $\delta'_\lambda$ is defined as
		$\inf\{t\in \R \colon \psi_\Gamma \leq t \lambda +\rho\}$ there.
		Here, we only consider $R'$ inside $\operatorname{int} \mathcal{L}_\Gamma^\star$
		where the definitions are equivalent.
	\end{remark}
	Finally, we recall the definition of $\mu_\Gamma\in \mathfrak{a}^\ast_+$
	as is used in \cite{LWW, limitcone}.
	It is unique element in $\mathfrak{a}^\ast_+$ satisfying
	\[
		\bigcap \{\operatorname{conv}(W \lambda) \colon
			\lambda\in \mathfrak{a}^\ast \text{ with }
		\Re \wt \sigma \subseteq \operatorname{conv}(W\lambda)\}
		= \operatorname{conv}(W \mu_\Gamma).
	\]
	We note that $\iota\mu_\Gamma =\mu_\Gamma$
	as well as
	$\wt \sigma \subseteq i \mathfrak{a}^\ast$ if and only if $\mu_\Gamma=0$.
	We refer to Figure~\ref{fig:mugamma} for a visualization
	of the definition for $\mu_\Gamma$.

	\begin{figure}
		\centering

		\includegraphics[width=\textwidth,trim=0cm 0.9cm 0.2cm 0.8cm, clip]
		{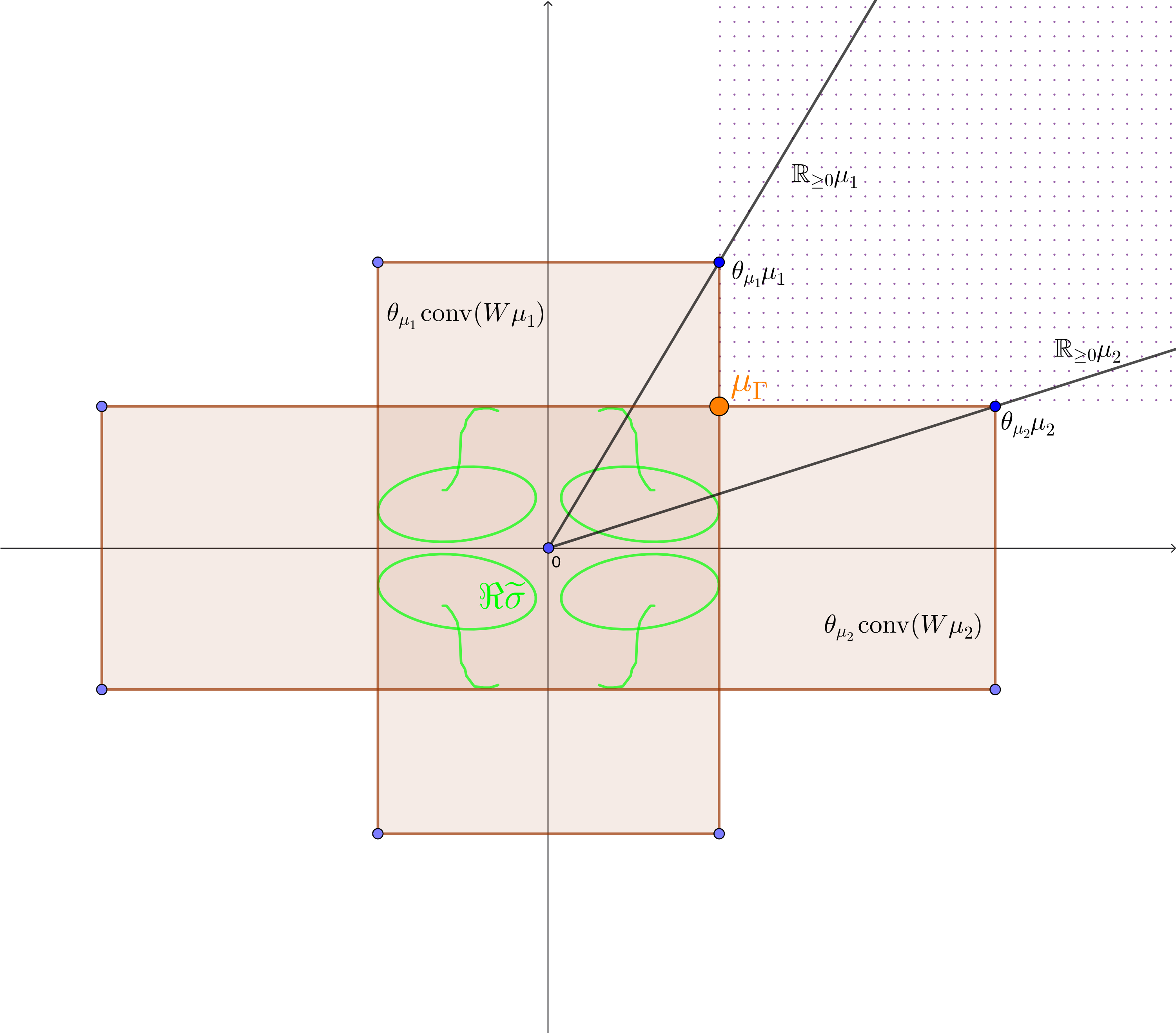}

		\caption{Definition of $\mu_\Gamma$ by the property that 
			$\operatorname{conv}(W\mu_\Gamma)$ is the smallest $W$-invariant
			convex hull that contains $\Re \wt \sigma$
			here depicted for the root system $A_1\times A_1$ where the Weyl group
			is generated by the reflections across the axes.
			$\theta_\mu \mu$ is the smallest element in $\R_{\geq 0} \mu$
			such that $\Re \wt \sigma$
			is contained in $\theta_\mu \operatorname{conv}(W\mu)$.
			Thus, the dotted region is the set $\{\mu \in \mathfrak{a}_+ \colon \Re \wt \sigma \subseteq \operatorname{conv}(W\mu)\}$.
		One should compare this to \eqref{eq:R'bymuGamma}.}
		\label{fig:mugamma}

	\end{figure}
	\subsection{Completion of the proof for Theorem~\ref{thmA}}%
	\label{sub:Completion of the proof for thmA}
	Recall that we showed
	\[
		\Re \wt \sigma \subseteq \max(0,\delta'_\lambda)
		\operatorname{conv}(W \lambda)
	\]
	or equivalently
	\[
		\theta_\lambda \leq \max(0,\delta'_\lambda)
	\]
	for $\lambda \in \mathfrak{a}_+^\ast$
	with $-\lambda \in W\lambda$
	in Sections~\ref{sec:analysisonGK}, \ref{sec:Separating tubes by invariant polynomials}, and \ref{sec:Spectralonlocally}
	if $\mathcal{L}_\Gamma \subseteq \mathfrak{a}_{++}$
	by analytical methods.
	In this section we will show the optimality of this result,
	i.e.~
	\begin{equation}
		\label{eq:optimality}
		\theta_\lambda = \max(0,\delta'_\lambda)
		\quad \text{for all }\lambda\in  \mathfrak{a}_+^\ast \text{ with } -\lambda \in W\lambda
	\end{equation}
	if $\Gamma$ is Zariski dense.
	This will finish our new proof of Theorem~\ref{thmA}.

	Let us first show that it suffices to see \eqref{eq:optimality}
	for $\mu_\Gamma$.

	\begin{proof}[Proof of \eqref{eq:optimality} from
		$\theta_{\mu_\Gamma}=\delta'_{\mu_\Gamma}$]
		Let us assume $\mu_\Gamma \neq 0$ so that $\delta_\lambda' >0$ for all $\lambda\in \mathfrak{a}_+^\ast$
		and $1 = \theta_{\mu_\Gamma} = \delta_{\mu_\Gamma}'$.
		In particular,
		$\mu_\Gamma \in R'$.
		By \eqref{eq:addlimitconeinR}
		\[
			\mu_\Gamma + {}_+\mathfrak{a}^\ast \subseteq R' = \bigcup_{\lambda\in \mathcal{L}_\Gamma^\star}
			[\delta'_\lambda, \infty) \lambda.
		\]
		On the other hand,
		for $\lambda\in \mathfrak{a}_+^\ast$,
		\[
			\operatorname{conv}(W \mu_\Gamma) \subseteq
			\theta_\lambda \operatorname{conv}(W \lambda)
		\]
		implies
		\[
			\theta_\lambda \lambda \in \mu_\Gamma + {}_+\mathfrak{a}^\ast
		\]
		by \cite[Lemma~IV.8.3]{gaga}.
		Thus, 
		\begin{align}
			\label{eq:R'bymuGamma}
			\mathfrak{a}^\ast_{+,\iota} \cap (\mu_\Gamma + {}_+ \mathfrak{a}^\ast)
			\subseteq
			\mathfrak{a}^\ast_{+,\iota} \cap R'
			\subseteq
			\bigcup_{\lambda\in \mathfrak{a}^\ast_{+,\iota}} [\delta'_\lambda,\infty) \lambda
			\subseteq
			\bigcup_{\lambda\in \mathfrak{a}^\ast_{+,\iota}} [\theta_\lambda,\infty) \lambda
			\subseteq
			\mathfrak{a}^\ast_{+,\iota} \cap (\mu_\Gamma + {}_+ \mathfrak{a}^\ast)
		\end{align}
		and we must have equality everywhere.
		This proves $\theta_\lambda = \delta'_\lambda$
		for all $\lambda\in 
		\mathfrak{a}^\ast_{+,\iota} \cap (\mu_\Gamma + {}_+ \mathfrak{a}^\ast)
		$
		from $\delta'_{\mu_\Gamma}=\theta_{\mu_\Gamma} = 1$.
	\end{proof}

	Let us now show that $\delta'_{\mu_\Gamma}=1$ to complete
	the new proof of Theorem~\ref{thmA}.
	For this we have to assume that $\Gamma$ is Zariski dense
	in the real algebraic group $G$.
	\begin{proof}
		[Proof of $\theta_{\mu_\Gamma} = \delta'_{\mu_\Gamma}$
		assuming $\Gamma$ is Zariski dense]
		The additional ingredient that we will use is a sharp bound on the spectrum
		of the Laplace-Beltrami operator 
		obtained by Anker and Zhang \cite{AZ22}.
		They showed
		\[
			\min (_\Gamma \Delta) = \|\rho\|^2 - \max (0,\delta')^2
		\]
		where
		\[
			\delta' \coloneqq \inf \left\{s\in \R\colon
			\sum_{\gamma\in \Gamma}^{} e^{-s\|\mu_+(\gamma)\|} <\infty\right\},
		\]
		see \cite{WZ23}.
		In contrast to Theorem~\ref{thm:spectrumsinglesa}
		this is a sharp bound on the spectrum.
		By \cite[Prop.~3.4]{limitcone} there is a unique normalized
		$\mu_\Delta \in \mathfrak{a}_+^\ast$
		such that $\delta' = \delta'_{\mu_\Delta}$.
		We can now copy the proof of \cite[Prop.~3.5]{limitcone}
		to get $\delta'_{\mu_\Delta} \mu_\Delta \in \wt \sigma$.
		We thus have
		\begin{align*}
			\delta'_{\mu_\Delta}\mu_\Delta \in \operatorname{conv}(W \mu_\Gamma)
			\subseteq
			\theta _{\mu_\Delta} \operatorname{conv}(W \mu_\Delta)
			\subseteq
			\delta'_{\mu_\Delta} \operatorname{conv}(W\mu_\Delta)
		\end{align*}
		by definition of $\mu_\Gamma$
		and $\theta _{\mu_\Delta} \leq \delta'_{\mu_\Delta}$.
		We infer
		$\mu_\Gamma \in W (\delta' _{\mu_\Delta}\mu_\Delta)$
		and $\delta'_{\mu_\Delta} = \theta _{\mu_\Delta}$
		which implies
		$\delta'_{\mu_\Gamma} =\theta_{\mu_\Gamma}$.
	\end{proof}

	To conclude this section let us formulate \eqref{eq:R'bymuGamma} as an independent result.
	\begin{proposition}
		\label{prop:R'inposchamber}
		\[
			\mathfrak{a}^\ast_{+,\iota} \cap R'
			=
			\mathfrak{a}^\ast_{+,\iota} \cap (\mu_\Gamma + {}_+ \mathfrak{a}^\ast)
		\]
	\end{proposition}
	We remark that Proposition~\ref{prop:R'inposchamber} also holds for $\mu_\Gamma =0$
	and we need neither $\mathcal{L}_\Gamma \subseteq \mathfrak{a}_{++}$
	nor that $\Gamma$ is Zariski dense
	if we use Theorem~\ref{thmA} directly in the proof.

	\subsection{Precise bounds on the domain of convergence}%
	\label{sub:Precise bounds on the domain of convergence}

	In this section we use Proposition~\ref{prop:R'inposchamber} and
	the convexity of $R$ (Lemma~\ref{la:Rconvex}) to bound $R'$ outside of $\mathfrak{a}^{\ast}_+$.
	The result we obtain is the following.
	\begin{proposition}
		\label{prop:upperboundR'}
		Let $I \coloneqq \{\alpha\in \Pi\colon \langle \alpha,\mu_\Gamma\rangle >0\}$.
		Then,
		\[
			R' \cap \mathfrak{a}^{\ast}_\iota \subseteq
			\mu_\Gamma + \sum_{\alpha\in I} \R_{\geq 0} \alpha + \sum_{\alpha\notin I} \R \alpha
			.
		\]
	\end{proposition}
	Let us shortly describe the idea in the case that $I=\Pi$,
	see Figure~\ref{fig:convexityargument}.
	For $\lambda\in R'\cap \mathfrak{a}_\iota^\ast$
	we consider the convex combination $\lambda_t \coloneqq\mu_\Gamma + t(\lambda-\mu_\Gamma)$ with $t\in [0,1]$.
	For $t\ll 1$, $\lambda_t$ is close to $\mu_\Gamma$
	and thus $\lambda_t \in \mathfrak{a}^\ast_{+,\iota}$.
	Proposition~\ref{prop:R'inposchamber} exactly determines $R'\cap \mathfrak{a}^\ast _{+,\iota}$
	and we get $\lambda_t \in \mu_\Gamma + {}_+ \mathfrak{a}^\ast$.
	Hence,
	\[
		\lambda -\mu_\Gamma \in {}_+ \mathfrak{a}^\ast
		=
		\sum_{\alpha\in \Pi} \R_{\geq 0} \alpha.
	\]

	\begin{figure}
		\centering
		\includegraphics[width=\textwidth,trim=0cm 1.8cm 2cm 1cm, clip]{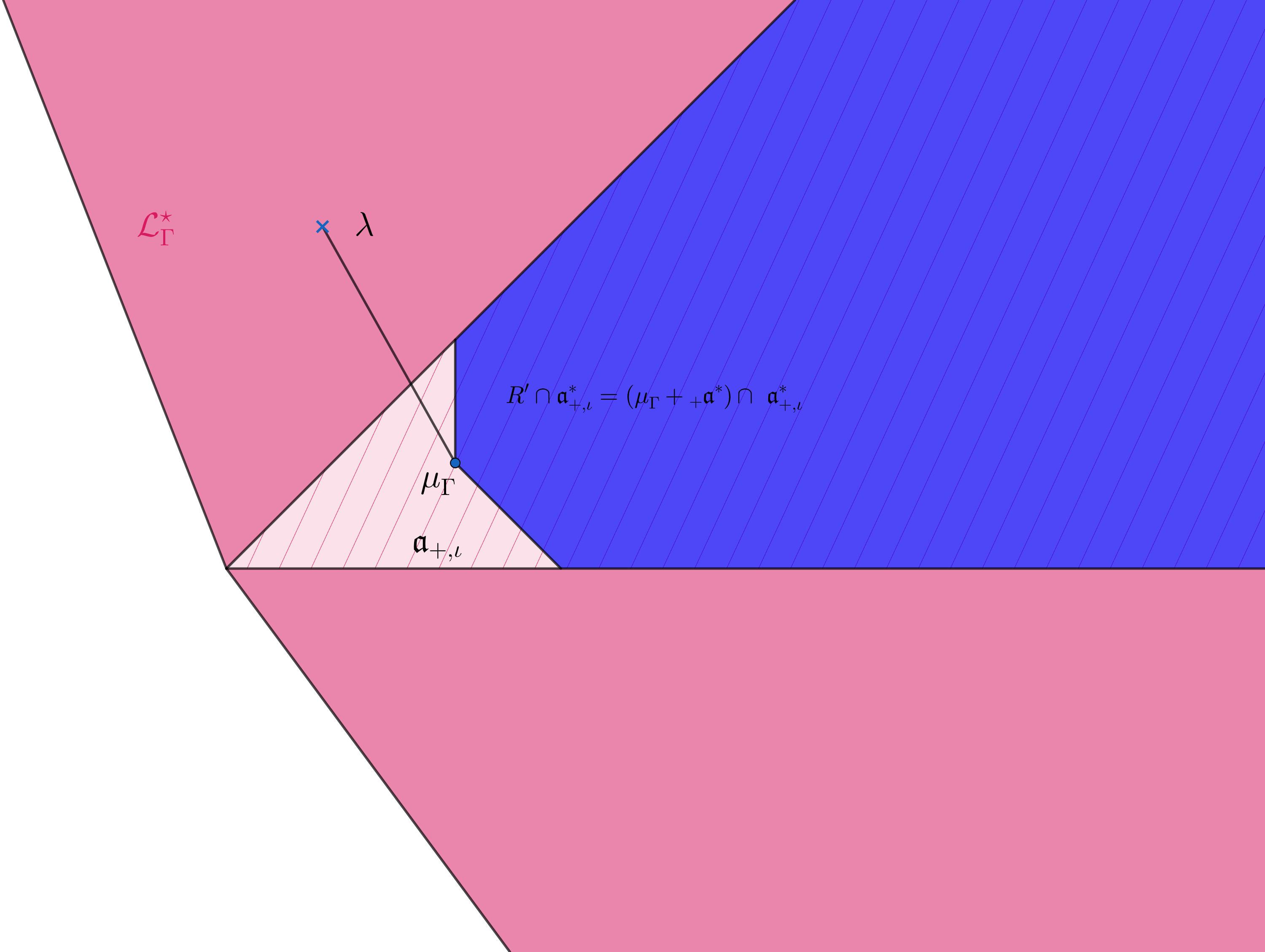}
		\caption{Visualization of the convexity argument used in Proposition~\ref{prop:upperboundR'}
			for the case that $\mu_\Gamma$ is regular.
			The whole line connecting $\lambda$ and $\mu_\Gamma$
			must be contained in $R'$.
			However, $R'$ in $\mathfrak{a}^\ast _{+,\iota}$ is given by the blue region.
		This gives a contradiction if $\lambda\notin \mu_\Gamma + {}_+ \mathfrak{a}^\ast$ and $\lambda\in \mathfrak{a}^\ast_\iota$.}
		\label{fig:convexityargument}
	\end{figure}
	For the general case this argument has to be refined in order to ensure $\lambda_t\in \mathfrak{a}^\ast_+$ for $t\ll 1$.

	\begin{proof}
		Let $\lambda = \iota \lambda\in R'$.
		We fix $\mu_I \in \mathfrak{a}^\ast$ such that
		$\mu_I(v_\alpha)=0$ for all $\alpha\in I$ and
		$\langle \mu_I,\beta\rangle =1$ for all $\beta\notin I$.
		This is possible as $v_\alpha$, $\alpha\in I$, together with $\Pi\setminus I$
		is a basis.
		We have $\mu_I(v_\beta)>0$ for $\beta\notin I$
		so that $\mu_I\in {}_+ \mathfrak{a}^\ast$.
		From this we infer $\mu_\Gamma + \varepsilon \mu_I\in R'$
		by \eqref{eq:addlimitconeinR}.
		Furthermore, $\mu_\Gamma +\varepsilon \mu_I \in \mathfrak{a}_{++}^\ast$
		for small $\varepsilon>0$.
		Since $I$ is $\iota$-invariant we also have $\mu_\Gamma +\varepsilon \mu_I \in \mathfrak{a}^{\ast}_{+,\iota}$.
		We now use the convexity of $R'$ to find
		\[
			t\lambda +(1-t)(\mu_\Gamma + \varepsilon \mu_I) \in R' \cap \mathfrak{a}^{\ast,\mathrm{Her}}_+
		\]
		for small $t>0$.
		The positivity follows because $\mu_\Gamma + \varepsilon \mu_I$ is in the interior of the positive chamber,
		and the $\iota$-invariance follows from the invariance of $\lambda$ and $\mu_\Gamma + \varepsilon \mu_I$.
		Proposition~\ref{prop:R'inposchamber} implies, for all $\alpha\in I$,
		\[
			\mu_\Gamma(v_\alpha) \leq 
			(
			t\lambda +(1-t)(\mu_\Gamma + \varepsilon \mu_I)
			)
			(v_\alpha)
			=\mu_\Gamma(v_\alpha) + t(\lambda(v_\alpha) - \mu_\Gamma(v_\alpha)).
		\]
		Dividing by $t$ gives
		\begin{equation}
			\label{eq:positivecoeff}
			\mu_\Gamma(v_\alpha) \leq \lambda(v_\alpha).
		\end{equation}
		The proposition follows by writing
		\[
			\lambda - \mu_\Gamma = \sum_{\alpha\in \Pi} c_\alpha \alpha 
		\]
		as \eqref{eq:positivecoeff} yields $c_\alpha \geq 0$ for $\alpha\in I$.
	\end{proof}
	Proposition~\ref{prop:upperboundR'} has some interesting consequences.
	Together with \eqref{eq:addlimitconeinR} it
	gives the following corollary.
	\begin{corollary}
		\label{cor:inclusion}
		\begin{align*}
			\mathcal{L}_\Gamma^\star \cap \mathfrak{a}_\iota^\ast
			\subseteq
			\sum_{\alpha\in I} \R_{\geq 0} \alpha + \sum_{\alpha\notin I} \R \alpha
		\end{align*}
		and by duality
		\begin{equation}
			\label{eq:inclinlimitcone}
			\sum_{\alpha\in I} \R_{\geq 0} (v_\alpha + v _{\iota\alpha})
			\subseteq
			\operatorname{conv}(\mathcal{L}_\Gamma).
		\end{equation}
		In particular, 
		\begin{equation}
			\label{eq:limitconeintersectsCalpha}
			C_\alpha \cap \mathcal{L}_\Gamma \neq \{0\}
			\qquad \forall \alpha\in I
			,
		\end{equation}
		where \[
			C_\alpha=\mathfrak{a}_+ \cap \bigcap_{\beta\in \Pi\setminus\{\alpha,\iota\alpha\}} \ker \beta
		= \R_{\geq 0} v_\alpha + \R_{\geq 0} v_{\iota \alpha}\]
		for $\alpha\in \Pi$.
	\end{corollary}
	\begin{proof}
		The only part to explain is that $v_\alpha +v _{\iota\alpha}\in \operatorname{conv}(\mathcal{L}_\Gamma)$
		implies $C_\alpha \cap \mathcal{L}_\Gamma \neq \{0\}$.
		However, this is true since the $v_\alpha$ are the extremal vertices
		of $\mathfrak{a}_+$ and $\mathcal{L}_\Gamma \subseteq \mathfrak{a}_+$.
	\end{proof}

	Let us consider the example of real rank $2$.

	\begin{example}
		Let us assume that $\iota$ is trivial
		and $G$ has real rank $2$, i.e.~$\Pi=\{\alpha,\beta\}$.
		There are essentially three cases.
		\begin{description}
			\item[The tempered case $I=\emptyset$] In this case $\mu_\Gamma=0$.
				We have
				\[
					\rho + \mathcal{L}_\Gamma^\star\subseteq R
					\subseteq\mathcal{L}_\Gamma^\star.
				\]
				There is no restriction on $\mathcal{L}_\Gamma$.
			\item[The non-tempered case $I=\{\alpha\}$]
				Since $\langle\mu_\Gamma,\beta\rangle=0$
				we must have
				$\mu_\Gamma=\delta'_{v_\alpha} v_\alpha$
				with $v_\alpha\in \mathfrak{a}^\ast_+$
				by identifying $\mathfrak{a}^\ast$ and $\mathfrak{a}$
				and $\delta'_{v_\alpha }>0.$
				Then,
				\[
					\delta'_{v_\alpha} v_\alpha + \mathcal{L}_\Gamma^\star
					\subseteq
					R'
					\subseteq
					\delta'_{v_\alpha} v_\alpha + \R _{\geq 0}\alpha + \R \beta.
				\]
				We must have $v_\alpha \in \mathcal{L}_\Gamma$.
			\item[The regular case $I=\{\alpha,\beta\}$] 
				In this case $\mu_\Gamma$ is regular.
				We can determine $R'$ completely.
				\[
					R' = \mu_\Gamma + \R_{\geq 0} \alpha + \R_{\geq 0} \beta
					.
				\]
				Moreover, $\mathcal{L}_\Gamma = \mathfrak{a}_+$.
		\end{description}
		See Figures~\ref{fig:temp}, \ref{fig:nontemp}, and \ref{fig:reg} for a visualization.
	\end{example}

	\begin{figure}
		\centering

		\includegraphics[width=\textwidth,trim=0cm 0.5cm 2cm 0.7cm, clip]{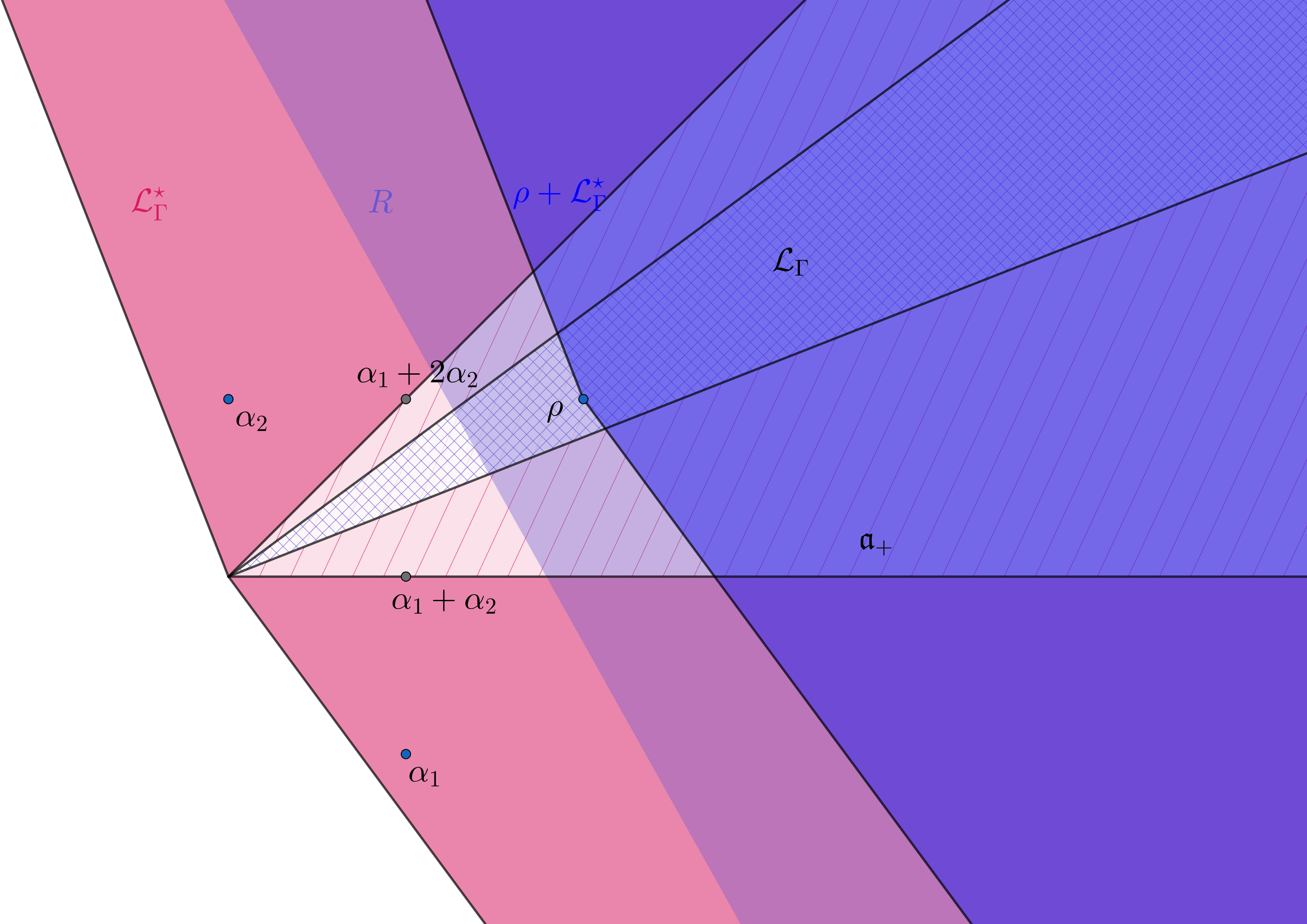}

		\caption{Tempered case}
		\label{fig:temp}

	\end{figure}
	\begin{figure}
		\centering

		\includegraphics[width=\textwidth,trim=0cm 0cm 3cm 2cm, clip]{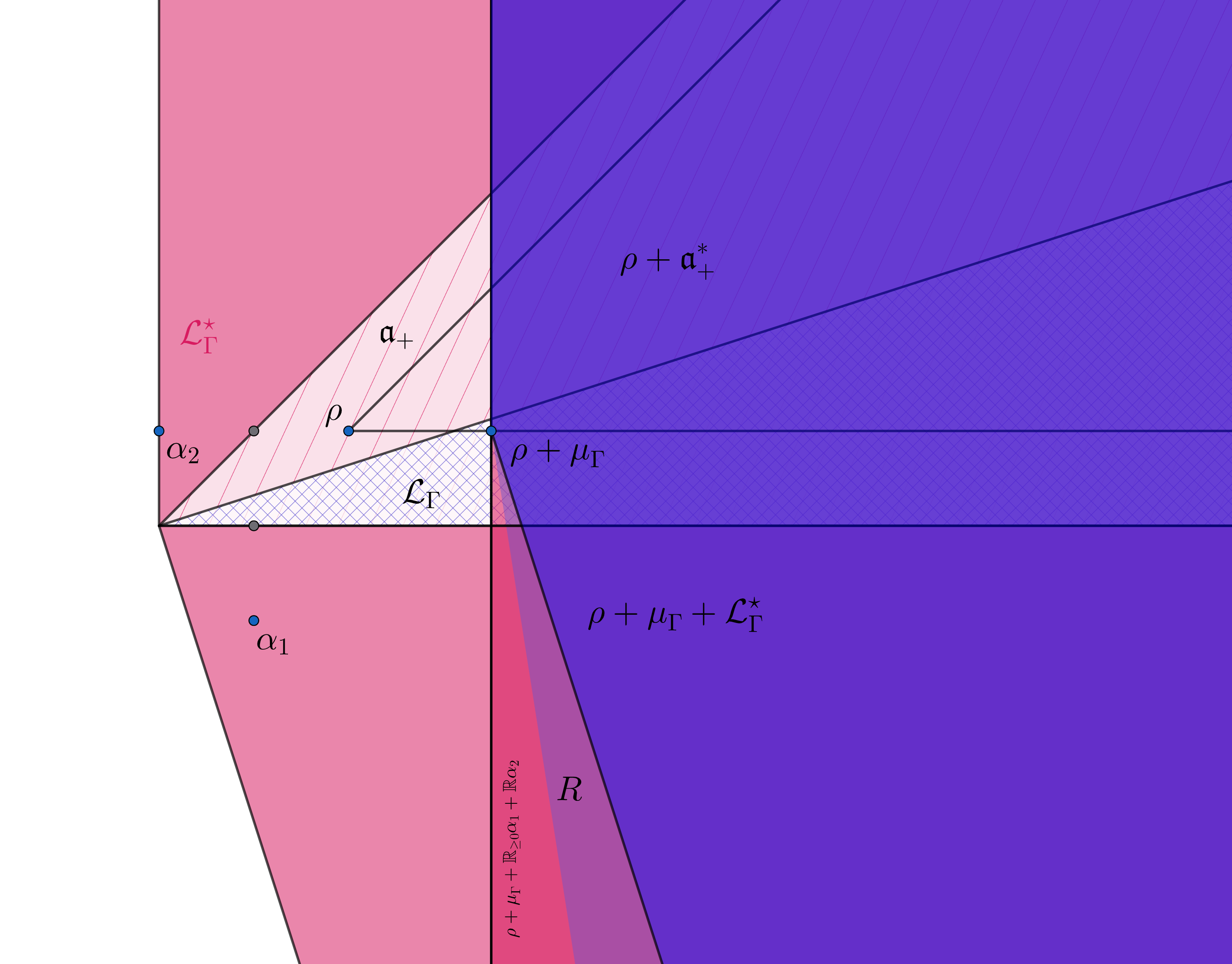}
		\caption{Non-tempered case}
		\label{fig:nontemp}

	\end{figure}
	\begin{figure}
		\centering

		\includegraphics[width=\textwidth,trim=0cm 2cm 2cm 2cm, clip]{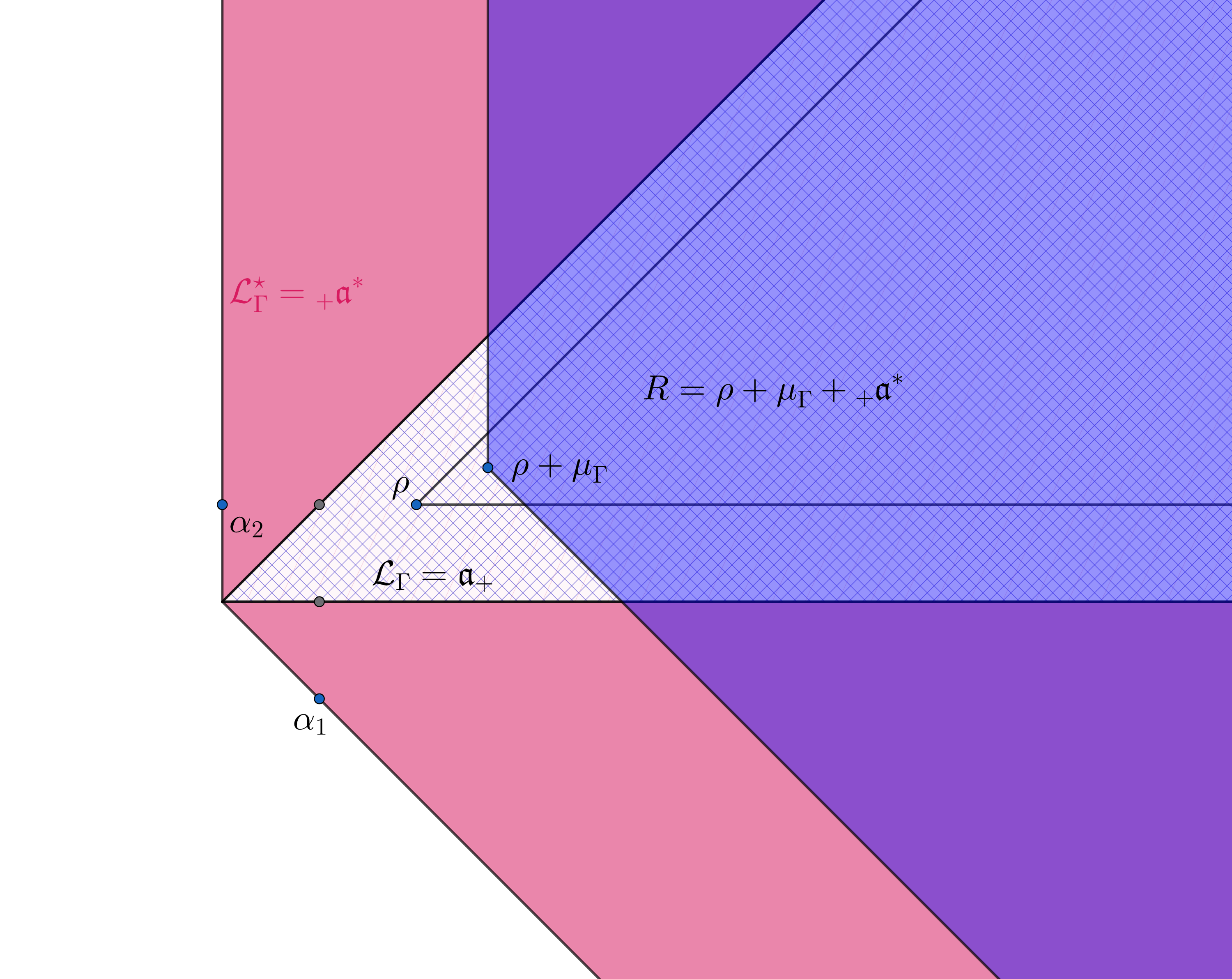}

		\caption{Regular case}
		\label{fig:reg}

	\end{figure}
	\begin{remark}
		Let us write first and third case in the usual generality.
		\begin{enumerate}
			\item $\mu_\Gamma=0$.
				Then
				\[
					\rho + \mathcal{L}_\Gamma^\star\subseteq R
					\quad\text{and}\quad
					R\cap \mathfrak{a}^\ast_\iota \subseteq\mathcal{L}_\Gamma^\star\cap \mathfrak{a}^\ast_\iota.
				\]
			\item $\mu_\Gamma$ is regular.
				Then
				\[\mu_\Gamma + {}_+\mathfrak{a}^\ast
					\subseteq
					R'
					\quad\text{and}
					\quad
					R'\cap \mathfrak{a}^\ast_\iota \subseteq  (\mu_\Gamma + {}_+\mathfrak{a}^\ast) \cap \mathfrak{a}_\iota^\ast.
				\]
				Moreover, $C_\alpha \cap \mathcal{L}_\Gamma\neq \{0\}$ 
				for all $\alpha\in \Pi$.
				Assuming that $\mathcal{L}_\Gamma$ is convex
				(e.g.~if $\Gamma$ is Zariski dense)
				then
				\[
					\mathcal{L}_\Gamma \cap \mathfrak{a}_{+,\iota}^\ast
					=
					\mathfrak{a}_{+,\iota}^\ast
					.
				\]
		\end{enumerate}
	\end{remark}

	We also remark that
	Corollary~\ref{cor:inclusion} implies the following version of \cite[Cor.~4.5(i)]{limitcone}:
	\begin{equation}
		\label{eq:onewallavoided}
		C_\alpha \cap \mathcal{L}_\Gamma =\{0\} \implies \langle\mu_\Gamma,\alpha\rangle =0.
	\end{equation}
	We also obtain Theorem~\ref{thmB} from \eqref{eq:onewallavoided}.
	Indeed, if $\mathfrak{g}$ is not one of $\mathfrak{sl}_3(\mathbb{K})$,
	$\mathbb{K}=\R,\C,\mathbb{H}$ or $\mathfrak{e}_6^{-26}$ nor of real rank one,
	then $C_\alpha$ is in the boundary of $\mathfrak{a}_+$ for all $\alpha\in \Pi$.
	Therefore $\mathcal{L}_\Gamma \subseteq \mathfrak{a}_{++}$
	implies $C_\alpha \cap \mathcal{L}_\Gamma = \{0\}$
	for all $\alpha\in \Pi$.
	By \eqref{eq:onewallavoided}, $\mu_\Gamma=0$
	which is equivalent to $\wt \sigma \subseteq i \mathfrak{a}^\ast$.

	\bibliographystyle{alpha}
	\bibliography{reference}

	\end{document}